\documentclass[a4paper,12pt,reqno,twosides,onecolumn,final]{amsart}

\usepackage[width=.99\textwidth]{caption}

\usepackage{amsmath,amsfonts,amssymb,amscd}

\usepackage[foot]{amsaddr} 

\usepackage[skip=8pt,indent=15pt]{parskip}

\usepackage{graphicx}

\usepackage{bm}  
\usepackage[left,pagewise]{lineno}

\usepackage{setspace}
\usepackage[utf8]{inputenc} 
\usepackage{lmodern}
\usepackage{subfig}

\usepackage[]{algorithm}
\usepackage[]{algorithmic}
\usepackage{amssymb}
\usepackage{mathabx}

\usepackage{mathtools}

\usepackage{color}
\usepackage{caption}

\DeclareMathOperator{\Div}{\mathrm{div}}
\DeclareMathOperator{\Kappa}{\mathcal{K}}

\let\div\relax

\DeclareMathOperator{\div}{div}

\DeclareMathOperator{\rot}{rot}
\DeclareMathOperator{\curl}{\mathbf{curl}}

\def\0{\phantom{0}}

\newtheorem{theorem}{Theorem}[section]
\newtheorem{lemma}[theorem]{Lemma}

\theoremstyle{definition}

\theoremstyle{remark}
\newtheorem{remark}[theorem]{Remark}

\begin{document}


\title[Unconditionally Stable Mixed Methods for Darcy Flow]{Unconditionally Stable Mixed Finite Element Methods for Darcy Flow}


\author[M.R. Correa]{Maicon R. Correa}
\address{Universidade Estadual de Campinas (UNICAMP), Departamento de Matemática Aplicada, IMECC, Brasil}
\curraddr{}
\email{maicon@ime.unicamp.br}
\thanks{}

\author[A.F.D. Loula]{Abimael F. D. Loula}
\address{LNCC - Laboratório Nacional de Computação Científica, Petrópolis, RJ, Brasil}
\curraddr{}
\email{aloc@lncc.br}
\thanks{}


\date{}

\dedicatory{\rm This is a preprint of the original article published in \\{\it Computer Methods in Applied Mechanics and Engineering, 197 (2008), pp 1525--1540.} \\ https://doi.org/10.1016/j.cma.2007.11.025 \\
DOI 10.1016/j.cma.2007.11.025}

\begin{abstract}
Unconditionally stable finite element methods for Darcy flow are
derived by adding least-squares residual forms of the governing
equations to the classical mixed formulations.
The proposed methods are free of mesh dependent stabilization
parameters and allow the use of the classical continuous Lagrangian
finite element spaces of any order for the velocity and the
potential.
Stability, convergence and error estimates are derived and numerical
experiments are presented to demonstrate the flexibility of the
proposed finite element formulations and to confirm the predicted
rates of convergence.
\end{abstract}

\maketitle



%
%
\section{Introduction}
\label{sec:intro}
%
%

The flow of an incompressible homogeneous fluid in a rigid saturated
porous media leads to the classical Darcy problem which basically
consists of the mass conservation equation plus Darcy's law, that
relates the average velocity of the fluid in a porous medium with
the gradient of a potential field through the hydraulic conductivity
tensor.  A standard way to solve this system of partial differential
equations is based on the second order elliptic equation obtained by
substituting Darcy's law in the mass conservation equation leading
to a single field problem in potential only with velocity calculated
by taking the gradient of the solution multiplied by the hydraulic
conductivity. Constructing finite element methods based on this kind
of formulation is straightforward. However, this direct approach
leads to lower-order approximations for velocity compared to
potential and, additionally, the corresponding balance equation is
satisfied in an extremely weak sense. Alternative formulations have
been employed to enhance the velocity approximation like mixed
methods \cite{RAVIART77,FORTIN91} and post-processing techniques
\cite{LOULA95,CORREA2007,CORDES92,DURLOFSKY94}.

Mixed methods are based on the simultaneous approximation of
potential and velocity fields using different spaces for velocity
and potential to satisfy a compatibility between the finite element
spaces, which reduces the flexibility in constructing stable finite
approximations (LBB condition - \cite{FORTIN91,BREZZI74}). One well
known successful approach is the Dual mixed formulation developed by
Raviart and Thomas \cite{RAVIART77} using divergence based finite
element spaces for the velocity field combined with discontinuous
Lagrangian spaces for the potential.
To overcome the compatibility condition, typical of mixed methods,
stabilized finite element methods have been developed
\cite{LOULA87A,LOULA88,FRANCA88,MASUD2002,BREZZI2005,HUGHES2006A,
LOULA2006A,CORREA2006,BARRENECHEA2007}. In \cite{LOULA88} stabilized
mixed Petrov-Galerkin methods are proposed for a heat transfer
problem identical to Darcy's problem, where the original saddle
point associated with the classical mixed formulation is converted
into a minimization problem. Non-symmetrical stabilized mixed
formulations for Darcy flow are presented in \cite{MASUD2002} by
adding an adjoint residual form of Darcy's law to the mixed Galerkin
formulation. More recently Petrov-Galerkin enriched methods
\cite{BARRENECHEA2007} as well as stabilized mixed formulations
associated with discontinuous Galerkin methods
\cite{BREZZI2005,HUGHES2006A} have been proposed and analyzed for
the same problem.

In this work we present a stabilized mixed finite element methods
derived by adding to the classical Dual mixed formulation least
square residuals of the Darcy's law, the mass balance equation and
the curl of Darcy's law. For sufficiently smooth hydraulic
conductivity, an unconditionally stable formulation is presented
with optimal rates of convergence in
$\left[H^1(\Omega)\right]^n\times H^1(\Omega)$ norm.

Some other possible stabilizations are commented, such as a Stokes
compatible $\left[H^1(\Omega)\right]^n\times L^2(\Omega)$
stabilization as well as ${H(\div,\Omega)}\times H^1(\Omega)$ and
$\left[L^2(\Omega)\right]^n\times H^1(\Omega)$ stable methods.
The equivalence between a symmetric
$\left[L^2(\Omega)\right]^n\times H^1(\Omega)$ stable method and the
non symmetric method proposed in \cite{MASUD2002} is observed. The
present stabilized methods are free of mesh dependent parameters and
all of them, except Stokes compatible mixed method, allow the use of
the classical continuous Lagrangian finite element spaces of any
order for velocity and potential, including equal-order
interpolations. The error estimates are numerically verified,
demonstrating the flexibility in the choice of the finite element
spaces, derived by the proposed stabilizations.

An outline of the paper follows.  In Section \ref{sec:model} we
present the model problem, introduce some notations and recall
standard single field and mixed Galerkin methods.  In Section
\ref{sec:newfem} we present and analyze the
$\left[H^1(\Omega)\right]^n\times H^1(\Omega)$ stable method. In
Section \ref{sec:otherfem} we comment on some alternative stabilized
formulations. Numerical experiments illustrating the performance of
the methods are reported in Section \ref{sec:num} and in Section
\ref{sec:conclu} we draw some conclusions.

%
%
\section{Model Problem and Galerkin Approximations}
\label{sec:model}
%
%

In this section we present our model problem, introduce some
notations and recall the standard single field and mixed Galerkin
methods as the starting point to the stabilized formulations
discussed next.
Let $\Omega$, a bounded, open, connected domain in
$\mathbb{R}^n$ ($n=2$ or $n=3$) with Lipschitz boundary
$\Gamma=\partial \Omega$ and outward unit normal vector $\bm{n}$, be
the domain of a rigid porous media saturated with an incompressible
homogeneous fluid.
Our model problem is formulated as follows: {\it Find the velocity
$\bm{u}$ and the hydraulic potential $p$ such that}
\begin{eqnarray}
\label{eq:darcy}
\bm{u}=-\Kappa\nabla p &\quad \mbox{ in } \; \Omega & \qquad \mbox{(Darcy's law)} \\
\label{eq:balanco}
\Div \, \bm{u} =f & \quad \mbox{ in } \;  \Omega & \qquad \mbox{(mass conservation)} \\
\bm{u}\cdot \bm{n}  = 0 & \quad \mbox{ on } \;
\partial \Omega, &   \qquad \mbox{(boundary condition)}
\end{eqnarray}
where, $\bm{u}$ is the average velocity of the fluid in the porous
media, related with the gradient of the hydraulic potential $p$
through the hydraulic conductivity tensor $\Kappa$. The source $f$ must
satisfy the compatibility condition
\begin{equation}
\label{eq:compat} \displaystyle \int_{\Omega} f d\Omega = 0.
\end{equation}
We assume that $\Kappa$ is a positive definite tensor,  so that there
exist constants $ 0 < \kappa_{\mathrm{min}}  \leq \kappa_{\mathrm{max}} $
such that
\begin{equation}
\label{eq:limk} \kappa_{\mathrm{min}} \xi^T \xi \leq \xi^T \Kappa \xi  \leq  \kappa_{\mathrm{max}}
\xi^T \xi
\end{equation}
for all $\xi\in \mathbb{R}^n$ and $K_{ij}\in L^\infty(\Omega)$.
Since $\Kappa$ is a positive definite tensor, its inverse
$\Lambda=\Kappa^{-1}$ called hydraulic resistivity tensor is well
defined and there exist constants $ 0 < \lambda_{\mathrm{min}} \leq \lambda_{\mathrm{max}}$
such that
\begin{equation}
\label{eq:liml} \lambda_{\mathrm{min}} \xi^T \xi \leq \xi^T \Lambda \xi  \leq  \lambda_{\mathrm{max}}
\xi^T \xi
\end{equation}
for all $\xi\in \mathbb{R}^n$. To guarantee further continuity
requirements, we assume that all elements of $\Lambda$ are
$C^1(\Omega)$ differentiable.

%
%
%
%
%
To present the standard Galerkin formulation we first introduce some
notations. Let $L^2(\Omega)$ be the space of square integrable
functions in $\Omega$ with inner product $(\cdot,\cdot)$ and norm
$\|\cdot \|$. $H^m(\Omega)$, for $m\geq 0$ integer, is the Hilbert
space of scalar functions in $\Omega$ with $m$ derivatives in
$L^2(\Omega)$ and has norm $\|\cdot\|_m$.
We also use the Hilbert space
\[
{H(\div,\Omega)}= \left\{\bm{u} \in \left[ L^2\Omega) \right]^n; \; \Div \bm{u} \in
L^2(\Omega)\right\},
\]
with norm
\[
\|\bm{u} \|_{H(\div,\Omega)} := \left\{\| \bm{u} \|^2 + \| \Div \bm{u} \|^2\right
\}^{1/2},
\]
and its subspace
\[
{H_0(\div,\Omega)}= \left\{\bm{u} \in {H(\div,\Omega)} \; ; \; \bm{u} \cdot \bm{n} = 0 \; \mbox{on} \;
\partial \Omega \right\}.
\]
Let $\curl = \nabla \times$ denote the curl operator, with the usual
identification of $\mathbb{R}^2$ with the $(x_1,x_2)$ plane in
$\mathbb{R}^3$ in the case $n=2$.
%
We introduce the Hilbert space
\[
{H(\curl_\Lambda,\Omega)}=  \left\{\bm{u} \in \left[ L^2\Omega) \right]^n; \; \curl
(\Lambda \bm{u}) \in \left[L^2(\Omega)\right]^{2n-3}\right\},
\]
with norm
\[
\|\bm{u} \|_{H(\curl_\Lambda,\Omega)}:= \left\{\| \bm{u} \|^2 + \| \curl (\Lambda \bm{u} )
\|^2\right \}^{1/2} .
\]

%
%
\subsection{Primal  Formulation}
\label{sec:primal}
%
%

A standard way to solve our model problem is based on the second
order elliptic equation obtained by substituting Darcy's law
(\ref{eq:darcy}) in the mass conservation equation
(\ref{eq:balanco})
\begin{equation}
\label{eq:poisson} -\Div( \Kappa \nabla p) = f \quad {in} \quad \Omega.
\end{equation}

To set the classical Primal variational formulation of Darcy flow,
let
\begin{equation}
\label{eq:q} \mathcal{P} = H^1(\Omega)/\mathbb{R}= \left\{ q\in H^1(\Omega)\,
; \, (q,1)=0 \right\}.
\end{equation}
Multiplying (\ref{eq:poisson}) by a function $q\in \mathcal{P}$ and
integrating it by parts over $\Omega$, we have: {\it
  Find $p \in \mathcal{P}$ such that for all $q \in \mathcal{P}$
  \begin{equation}
  \label{eq:pv}
    c(p,q)= f(q)
  \end{equation}
  with
   \begin{equation}
  c(p,q)=  (\Kappa \nabla p,\nabla q)
  \, ; \quad
  f(q)= (f,q).
   \end{equation}  %
}

Given $f\in H^{-1}(\Omega)$, existence and uniqueness of solution of
this problem is assured by Lax-Milgram Lemma \cite{CIARLET78}.

%
%
%

Let  $\{ {\mathcal T}_h \}$ be a family of partitions ${\mathcal
T}_h=\{\Omega^e\}$ of $\Omega$ indexed by the parameter $h$
representing the maximum diameter of the elements $\Omega^e
\in{\mathcal T}_h$.
Let $\mathcal{S}_h^{k}$ be the $C^0(\Omega)$ Lagrangian finite element space of
degree $k\geq 1$ in each element $\Omega^e$,
\begin{equation}
\label{eq:shk} \mathcal{S}_h^{k} = \{ \varphi_h \in C^0(\Omega); \left.
\varphi_h\right|_{\Omega^e} \in \mathbb{P}_k(\Omega^e) \}
\end{equation}
where ${\mathbb{P}}_k(\Omega^e)$ is the set of the polynomials of
degree $\leq k$ posed on $\Omega^e$.
Defining $ \mathcal{P}_h^{k} \subset \mathcal{P}$ as
\begin{equation}
\label{eq:qh1} \mathcal{P}_h^{k}=\mathcal{S}_h^{k}  \cap \mathcal{P},
\end{equation}
we have the classical Galerkin finite element approximation: {\it
  Find $p_h \in \mathcal{P}_h^{k}$ such that for all $q_h \in \mathcal{P}_h^{k}$
  \begin{equation}
    \label{eq:pvh}
    c(p_h,q_h)= f(q_h),
  \end{equation}
}
with its  well known approximation properties:
\begin{equation}
  \label{eq:errop}
  \| p - p_h \| + h \| \nabla p - \nabla p_h \| \leq C h^{k+1}
  \left| p \right|_{k+1} \, .
\end{equation}

 The velocity field can directly approximated through Darcy's law (\ref{eq:darcy})
\begin{equation}
  \label{eq:vg}
  \bm{u}_\mathrm{G} = -\Kappa \nabla p_h,
\end{equation}
yielding, from (\ref{eq:errop}), the following error estimate
\begin{equation}
  \label{eq:errovg}
  \| \bm{u} - \bm{u}_\mathrm{G} \| \leq  Ch^k | p|_{k+1}.
\end{equation}
This direct approximation for the velocity field is, in principle,
completely discontinuous on the interface of the elements and
presents very poor mass conservation and a too weak approximation of
flux (Neumann) boundary conditions.
The point is that the velocity is often the variable of main
interest.
Post-processing techniques have been proposed to improve both
velocity and mass conservation accuracy.
The basic idea of the post-processing formulations is to use the
optimal stability and convergence properties of the classical
Galerkin approximation (\ref{eq:pvh}) to derive more accurate
velocity approximations than that obtained by (\ref{eq:vg}).
In \cite{CORDES92} a post-processed continuous flux (normal
component of the velocity) distribution over the porous media domain
is derived using patches with flux-conserving boundaries. Another
conservative post-processing based on a control volume finite
element formulation is presented in \cite{DURLOFSKY94}. Within the
context of adaptive analysis,   a local post-processing, is
introduced in \cite{ZIENK92}, named Superconvergent Patch Recovery,
to recover higher order approximations for the gradients of finite
element solutions.
The Lagrangian based post-processings proposed in \cite{LOULA95} can
be extended to flows in heterogeneous porous media with an interface
of discontinuity of the hydraulic conductivity, as presented in
\cite{CORREA2007}.

%
%
\subsection{Dual Mixed Formulation}
\label{sec:dmf}
%
%

An alternative to compute more accurate velocity fields is to use
the dual mixed variational formulation:  {\it
Find $\{\bm{u},p\} \in \tilde{\mathcal{U}}\times \bar{\mathcal{P}}$ such that for all $\{\bm{v},q\} \in \tilde{\mathcal{U}}\times
\bar{\mathcal{P}}$
\begin{equation}
\label{eq:mpd} (\Lambda \bm{u},\bm{v}) - (\Div \bm{v},p) - (\Div \bm{u},q) =
-(f,q),
\end{equation}
}
with
\begin{equation}
\label{eq:ldz} \tilde{\mathcal{U}} = {H_0(\div,\Omega)},
\end{equation}
\begin{equation}
\bar{\mathcal{P}}=L^2(\Omega)/\mathbb{R}=\left\{ q\in L^2(\Omega) \; ; \;
(q,1)=0\right\}.
\end{equation}

This formulation characterizes a saddle point
\cite{FORTIN91,ODEN4,GIRAULT86}. Existence and uniqueness of
solution for this mixed formulation is assured by Brezzi's Theorem
\cite{BREZZI74}.
A well known successful Dual mixed finite element formulation was
introduced by Raviart and Thomas \cite{RAVIART77} using divergence
based finite element spaces for the velocity field combined with
discontinuous Lagrangian spaces for the potential. Many authors call
this method simply as Mixed Finite Element Method, but other stable
Galerkin approximations for the Dual mixed formulation have been
derived, as that proposed in \cite{BREZZI85}, for example.

%
%
\subsection{Primal Mixed Formulation}
\label{sec:pmf}
%
%

Another mixed formulation is derived by integrating by parts the
residual form of equation {\rm (\ref{eq:darcy})} yielding:
 {\it Find $\{\bm{u},p\} \in
\bar{\mathcal{U}}\times\mathcal{P}$ such that, for all $\{\bm{v},q\} \in \bar{\mathcal{U}}\times\mathcal{P}$},
\begin{equation}
\label{eq:mpf} (\Lambda \bm{u},\bm{v}) + (\bm{v},\nabla p) + (\bm{u},\nabla q)=
-(f,q)
\end{equation}
with $\bar{\mathcal{U}} = \left[L^2(\Omega)\right]^n$.
Stable Lagrangian finite element methods can be obtained by
choosing, for example,  discontinuous approximations for the
velocity combined with {\rm (\ref{eq:qh1})} for the potential, with
same order interpolations for both fields. The use of discontinuous
spaces for the velocity allows condensing its degrees of freedom at
the element level, resulting in a positive-definite algebraic system
for potential only. In general, stable methods for this formulation
lead to the same convergence rates obtained for single field
Galerkin formulation.

%
%
\section{ Unconditionally Stable Formulation}
\label{sec:newfem}
%
%

In this Section we present a stabilized formulation for Darcy
problem, unconditionally stable in $\left[H^1(\Omega)\right]^n\times
H^1(\Omega)$ and with no mesh dependent parameters, such that any
conforming finite element approximation is stable.
Though the formulations presented here are clearly applicable to
more general situations, in what follows we will consider only
isotropic media, i.e., we admit that the hydraulic conductivity
$\Kappa(\bm{x})$ and the hydraulic resistivity $\Lambda(\bm{x})$ are
functions only of the position given by
\[\Kappa(\bm{x})=\kappa(\bm{x}) \, \mathrm{\bf I} \quad \mbox{and} \quad
\Lambda(\bm{x})= \lambda(\bm{x}) \, \mathrm{\bf I}.
\]

Let now
\begin{equation}
\mathcal{U} ={H_0(\div,\Omega)} \cap H(\curl_{\Lambda},\Omega)
\end{equation}
and $\mathcal{P}$ be the functional spaces for velocity and potential,
respectively. The norm in $\mathcal{U}$ is defined by
\[
\|\bm{u}\|_{\mathcal{U}}^{2} = \|\bm{u}\|^2 + \|\Div \bm{u}\|^2 + \|\curl \Lambda \bm{u} \|^2.
\]

The following Lemma establishes the equivalence between the norms of
the spaces $\mathcal{U}$ and $[H^1(\Omega)]^n$.

\begin{lemma}
\label{lem:eq} The norms $\|\cdot\|_{\mathcal{U}}$ and $\|\cdot\|_{1}$ are
equivalent in $\mathcal{U}$; that is, for any $\bm{v}\in \mathcal{U}$ there exist
constants $\alpha_1$ and $\alpha_2$ such that
\begin{equation}
\label{eq:eq} \alpha_1 \left\|\bm{v}\right\|_1 \leq \|\bm{v}\|_{\mathcal{U}}
\leq\alpha_2 \left\|\bm{v}\right\|_1
\end{equation}
\end{lemma}

\begin{proof} See reference \cite{CAI97}. In the proof of this
Theorem  $C^1(\Omega)$ continuity of $\lambda$ is assumed.
\end{proof}

%
%
\subsection{The proposed Formulation}
%
%

Based on Lemma \ref{lem:eq}, we look for a stabilized formulation in
the product space $\mathcal{U}\times\mathcal{P}$ with norm
\begin{equation}
\label{eq:cpf} \| \{\bm{u},p\} \|_{\mathcal{U}\times\mathcal{P}}^2 = \| \bm{u} \|^2 + \| \Div \bm{u} \|^2 + \|
\curl \Lambda \bm{u} \|^2 + \|\nabla p \|^2 \,.
\end{equation}

Subtracting the least square residual of the Darcy's law
(\ref{eq:darcy}) and adding the least square residuals of the mass
balance equation (\ref{eq:balanco}) and of the curl of Darcy's law
to the Dual mixed formulation (\ref{eq:mpd}), we derive the
following symmetric formulation

\noindent{ \bf Problem CGLS:} {\it Find $\{\bm{u},p\} \in \mathcal{U}\times\mathcal{P}$ such that, for all
$\{\bm{v},q\}\in \mathcal{U}\times\mathcal{P}$
\begin{equation}
  \label{eq:cgls}
  B_{\mathrm{CGLS}} \left( \{\bm{u},p\};\{\bm{v},q\}\right) = F_{\mathrm{CGLS}}\left( \{\bm{v},q\}\right)
\end{equation}
with
\begin{eqnarray}
\nonumber B_{\mathrm{CGLS}}\left( \{\bm{u},p\}  ;  \{\bm{v},q\} \right)&=&
(\lambda\bm{u},\bm{v})- (\Div \bm{v}, p) - (\Div \bm{u}, q) \\ \nonumber && -
\frac{1}{2}\left( K\left( \lambda \bm{u} + \nabla p\right), \lambda \bm{v}
+ \nabla q \right) \\ \nonumber
&& + \frac{1}{2}\left(\lambda\Div \bm{u},\Div \bm{v} \right) \\
\label{eq:g-bilinear} && +\frac{1}{2}\left(K\curl(\lambda \bm{u}),\curl
(\lambda \bm{v}) \right)
\end{eqnarray}
\begin{equation}
\label{eq:g-linear} F_{\mathrm{CGLS}}\left(\{\bm{v},q\}\right) =  - (f,q) +
\frac{1}{2}(\lambda f, \Div\bm{v}).
\end{equation}
}

The analysis of this formulation show that each particular least
square residual considered contributes to improve the stability of
the dual mixed formulation in the following way:

\begin{itemize}
\item Darcy's law: $H^1(\Omega)$ stability to the potential;
\item Mass balance: $H(\div)$ stability to the velocity;
\item Curl of Darcy's law: combined with the least square residual
of the mass balance equation, enhances the $H(\div)$  stability of the
velocity to $\left[H^1(\Omega)\right]^n$.
\end{itemize}

\begin{theorem}
\label{teo:cp}
 {  Problem {\rm CGLS}} has a unique solution.
\end{theorem}

To prove Theorem \ref{teo:cp} we will use two fundamental lemmas. We
start by proving the boundedness, or continuity, of the bilinear
form $B_{\mathrm{CGLS}}(\cdot,\cdot)$.

\begin{lemma}[Continuity]
\label{lem:cp-cont}
There exists a positive constant $M$ such that for all $\{\bm{u},p\}\in \mathcal{U}\times\mathcal{P}$
and $\{\bm{v},q\}\in \mathcal{U}\times\mathcal{P}$ we have
\begin{equation}
\label{eq:cp-cont}
 |B_{\mathrm{CGLS}}(\{\bm{u},p\},\{\bm{v},q\})|\leq  M \| \{\bm{u},p\} \|_{\mathcal{U}\times\mathcal{P}}
\| \{\bm{v},q\} \|_{\mathcal{U}\times\mathcal{P}}
\end{equation}
\end{lemma}

\begin{proof}
Integrating $B_{\mathrm{CGLS}}(\cdot,\cdot)$ by parts,
we get
\begin{eqnarray*}
|B_{\mathrm{CGLS}}(\{\bm{u},p\},\{\bm{v},q\})| &=& \frac{1}{2}\left|
\left(\lambda\bm{u},\bm{v}\right) + \left(\bm{v},\nabla p\right)
+ \left( \bm{u},\nabla q\right) -\left(K \nabla p,\nabla q \right)\right. \\
&&\left.+ \left( \lambda \Div \bm{u} , \Div \bm{v}\right) +  \left( K
\curl (\lambda \bm{u}) , \curl (\lambda \bm{v})\right)\right|.
\end{eqnarray*}
Using Cauchy-Schwarz inequality:
\begin{eqnarray*}
|B_{\mathrm{CGLS}}(\{\bm{u},p\},\{\bm{v},q\})| &\leq& \frac{1}{2}\left[\lambda_{\mathrm{max}}\|\bm{u}\|
\|\bm{v}\| + \|\bm{v}\| \|\nabla p\|+\|\bm{u}\| \|\nabla q\|  \right. \\ &&+
\kappa_{\mathrm{max}} \|\nabla p\| \|\nabla q\|+\lambda_{\mathrm{max}}\|\Div \bm{u}\| \| \Div \bm{v}\|\\
&&\left. + \kappa_{\mathrm{max}}\|\curl \lambda \bm{u}\| \| \curl \lambda \bm{v}\|\right]\\
&\leq& M_1 \left\{ \|\bm{u}\|+\|\Div\bm{u}\|+\|\curl \lambda \bm{u}\| \right.
\\ && \left.+\|p\| +\|\nabla p\|\right\}\left\{ \|\bm{v}\|
+\|\Div\bm{v}\|+\|\curl \lambda \bm{v}\| \right. \\&&
\left.+\|q\|+\|\nabla q\| \right\}.
\end{eqnarray*}
Noting that
\[
\|\bm{u}\| + \|\Div \bm{u} \| + \|\curl \lambda\bm{u} \| + \|p\| +\|\nabla
p\| \leq \sqrt{2}\|\{\bm{u},p\}\|_{\mathcal{U}\times\mathcal{P}}
\]
we have (\ref{eq:cp-cont}), with
\( M=\max \{\lambda_{\mathrm{max}},\kappa_{\mathrm{max}}\}\).
\end{proof}

The next Lemma establishes the stability in the sense of Bab\v uska
\begin{lemma}[Stability]
\label{lem:cp-coerc}
There is a positive constant $\alpha$ such that for all $\{\bm{u},p\}\in \mathcal{U}\times\mathcal{P}$
we have
\begin{equation}
\label{eq:cp-coerc}
 \sup_{\{\bm{v},q\}\in \mathcal{U}\times\mathcal{P}}\frac{B_{\mathrm{CGLS}}(\{\bm{u},p\},\{\bm{v},q\})}{\|\{\bm{v},q\}\|_{\mathcal{U}\times\mathcal{P}}}
 \geq  \alpha \| \{\bm{u},p\} \|_{\mathcal{U}\times\mathcal{P}} .
\end{equation}
\end{lemma}

\begin{proof} To prove Lemma \ref{lem:cp-coerc}, we use the fact that
we can always choose
\[
\left\{ \bar{\bm{v}}, \bar{q} \right\} = \left\{ \bm{u}, -p \right\}
\]
such that for all $\{\bm{u},p\}\in \mathcal{U}\times\mathcal{P}$
\begin{eqnarray*}
 \frac{\,B_{\mathrm{CGLS}}(\{\bm{u},p\},\left\{ \bar{\bm{v}}, \bar{q} \right\} )}
 {\|\left\{ \bar{\bm{v}}, \bar{q} \right\} \|_{\mathcal{U}\times\mathcal{P}}}
 & = &  \frac{\|\lambda^{1/2} \bm{u}\|^2
 +\|K^{1/2} \nabla p\|^2}
 {2\|\{\bm{u},p\} \|_{\mathcal{U}\times\mathcal{P}}}\\
 &  & + \; \frac{ \|\lambda^{1/2} \Div \bm{u}\|^2
 +\|K^{1/2}\curl(\lambda \bm{u})\|^2}
 {2\|\{\bm{u},p\} \|_{\mathcal{U}\times\mathcal{P}}}\\
 & \geq &  \frac{\lambda_{\mathrm{min}} \| \bm{u}\|^2+ \kappa_{\mathrm{min}} \| \nabla p\|^2}
 {2\|\{\bm{u},p\} \|_{\mathcal{U}\times\mathcal{P}}}\\
 &  & + \; \frac{ \lambda_{\mathrm{min}} \|\Div \bm{u}\|^2
 +  \kappa_{\mathrm{min}} \|\curl(\lambda \bm{u})\|^2}
 {2\|\{\bm{u},p\} \|_{\mathcal{U}\times\mathcal{P}}}\\
 & \geq & \frac{\alpha \|\{\bm{u},p\}\|_{\mathcal{U}\times\mathcal{P}}^2}{\|\{\bm{u},p\} \|_{\mathcal{U}\times\mathcal{P}}}
  =   \alpha \| \{\bm{u},p\} \|_{\mathcal{U}\times\mathcal{P}} ,
\end{eqnarray*}
with $\alpha=\frac{1}{2}\min\left\{\lambda_{\mathrm{min}},\kappa_{\mathrm{min}}\right\}$.
Thus
\begin{eqnarray*}
 \sup_{\{\bm{v},q\}\in \mathcal{U}\times\mathcal{P}}\frac{|B_{\mathrm{CGLS}}(\{\bm{u},p\},\{\bm{v},q\})|}{\|\{\bm{v},q\}\|_{\mathcal{U}\times\mathcal{P}}}
 \geq  \frac{|\,B_{\mathrm{CGLS}}(\{\bm{u},p\},\left\{ \bar{\bm{v}}, \bar{q} \right\} )|}
 {\|\left\{ \bar{\bm{v}}, \bar{q} \right\} \|_{\mathcal{U}\times\mathcal{P}}}
 &\geq&  \alpha \| \{\bm{u},p\} \|_{\mathcal{U}\times\mathcal{P}}
\end{eqnarray*}
for all $\{\bm{u},p\} \in \mathcal{U}\times\mathcal{P}$. \end{proof}

\begin{proof}[Proof of Theorem {\rm \ref{teo:cp}}]
 Letting $f\in
L^2(\Omega)$ we have the continuity of the linear form
$F_{\mathrm{CGLS}}(\cdot)$, and the proof follows immediately from
Lemma \ref{lem:cp-cont}, Lemma \ref{lem:cp-coerc} and the Babu\v ska
Lemma \cite{BABUSKA71A}. 
\end{proof}

%
%
\subsection{Finite Element Approximation}
%
%

Defining the discrete space for the velocity
\begin{equation}
\label{eq:uh1}
 \mathcal{U}_h^{l} = [\mathcal{S}_h^l]^n \cap \mathcal{U}
\end{equation}
and using $\mathcal{P}_h^{k}$ as in (\ref{eq:qh1}) for the potential, we introduce
the following finite element approximation based on the proposed
CGLS formulation:

\noindent{ \bf Problem CGLSh:} {\it Find $\{\bm{u}_h,p_h\} \in \mathcal{U}_h^{l}\times\mathcal{P}_h^{k}$ such that
\[
  B_{\mathrm{CGLS}} \left( \{\bm{u}_h,p_h\};\{\bm{v}_h,q_h\}\right) = F_{\mathrm{CGLS}}\left(
  \{\bm{v}_h,q_h\}\right) \ \ \forall \{\bm{v}_h,q_h\}\in \mathcal{U}_h^{l}\times\mathcal{P}_h^{k} .
\]
}

The unconditional stability and convergence of the CGLSh method are
guaranteed by the following result:
\begin{theorem}
\label{teo:cph} Suppose  that the solution to {\rm (\ref{eq:cpf})}
satisfies $\bm{u}\in \left[H^{l+1}(\Omega)\right]^n$ and $p\in
H^{k+1}(\Omega)$, then {\rm Problem CGLSh} has a unique solution
satisfying
\begin{equation}
\label{eq:cp-est} \left\|\bm{u} - \bm{u}_h \right\|_1 + \left\| p-p_h
\right\|_1 \leq C\left( h^l \left| \bm{u} \right|_{l+1} + h^k \left| p
\right|_{k+1} \right).
\end{equation}
\end{theorem}

This Theorem shows  that the new method leads to optimal convergence
rates in $\left[ H^1(\Omega) \right]^n$ for velocity, for $l\le k$,
and optimal rates for both velocity and pressure if $l=k$.
Before proving  Theorem {\ref{teo:cph}}, we consider some
preliminary results.

\begin{lemma}
\label{lem:bound1}  Let $\{\bm{u},p\}$ and $\{\bm{u}_h,p_h\}$ be the solutions of {\rm
Problem CGLS} and {\rm Problem CGLSh}, respectively. Then the
following estimative holds:
\begin{equation}
\label{eq:bound1} \| \{\bm{u}-\bm{u}_h,p-p_h\} \|_{\mathcal{U}\times\mathcal{P}} \leq {\frac{M}{\alpha}} \| \{\bm{u}-\bm{v}_h,p-q_h\}
 \|_{\mathcal{U}\times\mathcal{P}} \quad \forall \; \{\bm{v}_h,q_h\} \in \mathcal{U}_h^{l}\times\mathcal{P}_h^{k}.
\end{equation}
\end{lemma}

\begin{proof} Since the approximation is conform ($\mathcal{U}_h^{l} \subset \mathcal{U}$
and $\mathcal{P}_h^{k}\subset\mathcal{P}$), choosing $\{\bar{\bm{v}},\bar{q}\}=\{\bm{u},-p\}$ and
$\{\bar{\bm{v}_h},\bar{q_h}\}=\{\bm{u}_h,-p_h\}$ we have
\begin{equation}
 \label{eq:mgsc1}
 \alpha \|\{\bm{u}-\bm{u}_h,p-p_h\} \|_{\mathcal{U}\times\mathcal{P}} \leq
 \frac{B_{\mathrm{CGLS}}\left(\{\bm{u},p\}-\{\bm{u}_h,p_h\},\{\bar{\bm{v}},\bar{q}\}-\{\bar{\bm{v}}_h,\bar{q}_h\}
 \right)}{\|\{\bar{\bm{v}},\bar{q}\}-\{\bar{\bm{v}}_h,\bar{q}_h\}\|_{\mathcal{U}\times\mathcal{P}}}.
\end{equation}
Using the variational consistency of the method we get the
orthogonality result
\begin{equation}
\label{eq:mpeort} B_{\mathrm{CGLS}}(\{\bm{u}-\bm{u}_h,p-p_h\}, \{\bm{v}_h,q_h\}) = 0 \quad \forall
\; \{\bm{v}_h,q_h\} \in \mathcal{U}_h^{l}\times\mathcal{P}_h^{k},
\end{equation}
that leads to
\begin{eqnarray}
\label{eq:mgsc2}
B_{\mathrm{CGLS}}\left(\{\bm{u}-\bm{u}_h,p-p_h\},\{\bar{\bm{v}},\bar{q}\}-\{\bar{\bm{v}}_h,\bar{q}_h\}\right)=
\hspace{3.87cm} \nonumber \\ B_{\mathrm{CGLS}}\left(\{\bm{u}-\bm{u}_h,p-p_h\}
,\{\bar{\bm{v}},\bar{q}\}-\{{\bm{v}}_h,-{q}_h\} \right) \, \forall \; \{\bm{v}_h,q_h\}
\in \mathcal{U}_h^{l}\times\mathcal{P}_h^{k}.
\end{eqnarray}
Substituting  (\ref{eq:mgsc2}) in (\ref{eq:mgsc1}) e using the
continuity (Lemma \ref{lem:cp-cont}) we have:
\begin{equation}
\label{eq:mpc3} \alpha \|\{\bm{u}-\bm{u}_h,p-p_h\} \|_{\mathcal{U}\times\mathcal{P}} \leq
\frac{M\|\{\bm{u}-\bm{u}_h,p-p_h\}\|_{\mathcal{U}\times\mathcal{P}}\,\|\{\bar{\bm{v}},\bar{q}\}-\{{\bm{v}}_h,-{q}_h\}\|_{\mathcal{U}\times\mathcal{P}}}
{\|\{\bar{\bm{v}},\bar{q}\}-\{\bar{\bm{v}}_h,\bar{q}_h\}\|_{\mathcal{U}\times\mathcal{P}}}.
\end{equation}
Considering that
\[
{\|\{\bar{\bm{v}},\bar{q}\}-\{\bar{\bm{v}}_h,\bar{q}_h\}\|_{\mathcal{U}\times\mathcal{P}}}=
\|\{\bm{u},p\}-\{\bm{u}_h,p_h\}\|_{\mathcal{U}\times\mathcal{P}}
\]
we get the estimative (\ref{eq:bound1}).
\end{proof}

\begin{proof}[Proof of Theorem {\rm \ref{teo:cph}}]. The proof of the
existence of the solution follows immediately from Theorem
\ref{teo:cp}. Since the stability  was proved in Lemma
\ref{lem:cp-coerc} independently of any compatibility condition
between the spaces $\mathcal{U}$ and $\mathcal{P}$,  CGLS is a consistent and
unconditionally stable formulation in the sense of Babu\v ska, which
means that any conform finite element approximation is stable.
The same holds for the continuity. Therefore,  existence and
uniqueness of the solution of the {  Problem CGLSh} follow as a
Corollary of Theorem \ref{teo:cp}.
Finally, the error estimate (\ref{eq:cp-est}) follows immediately
from Lemmas \ref{lem:bound1} and \ref{lem:eq}.
\end{proof}

\begin{remark}
Error estimates based on  Bab\v uska Lemma usually have stability
constant $(1+M / \alpha)$. The constant $M/ \alpha$ of estimate
{\rm(\ref{eq:bound1})}, typical of estimates based on Lax-Milgran
Lemma, is a consequence of the unconditional stability of the
proposed formulation.
\end{remark}

%
%
\subsection{Error Estimate in $[L^2(\Omega)]^2\times L^2(\Omega)$ Norm}
%
%

Under appropriate regularity assumptions, convergence rates in
$[L^2(\Omega)]^2\times L^2(\Omega)$ for velocity and potential are
proved for Problem CGLS.
For simplicity, let us consider homogeneous porous media
($K=$constant) and bidimensional domain $\Omega\in \mathbb{R}^2$. In
this case it will be useful to define the curl of a scalar field as
the vector field
\[
\curl q = \left(
\begin{array}{r}
-\partial q/\partial x_2 \\
 \partial q/\partial x_1
\end{array}
\right)
\]
and the rotational of a vector field as the  scalar field given by
\[
\rot \bm{v} = \frac{\partial v_1}{\partial x_2 } - \frac{\partial
v_2}{\partial x_1 }.
\]
With these notations, the CGLS formulation can be written as: {\it
Find $\{\bm{u},p\}\in\mathcal{U}\times\mathcal{P}$ such that for all $\{\bm{v},q\}\in\mathcal{U}\times\mathcal{P}$}
\begin{eqnarray*}
\lambda(\bm{u},\bm{v}) -(p,\Div \bm{v}) -(q,\Div \bm{u}) &+& \lambda(\Div \bm{u}
,\Div \bm{v}) + \lambda (\rot \bm{u},\rot\bm{v})\\ &-& \kappa(\nabla p,\nabla q) =
-2(f,q)+\lambda(f,\Div\bm{v}) .
\end{eqnarray*}
Integrating by parts and using the following integration formula
\begin{equation}
\int_\Omega \curl q \cdot \bm{v} d\Omega = \int_\Omega q \rot \bm{v}
d\Omega + \int_{\partial \Omega} q \bm{v} \cdot \tau \,d\Gamma
\end{equation}
where $\tau$ is the unit tangent vector in the counterclockwise
direction on $\partial \Omega$, we get
\begin{eqnarray*}
\left( \lambda \bm{u} + \nabla p -\lambda \nabla \Div \bm{u} + \lambda
\curl \rot \bm{u} + \lambda \nabla f,\bm{v} \right) + (K\Delta p -\Div \bm{u}
+2f,q) \\+ \int_{\partial \Omega}\left( p+\lambda \Div \bm{u} - \lambda
f\right)\bm{v}\cdot \bm{n} \, d\Gamma \\ -\lambda \int_{\partial \Omega}
(\rot \bm{u})\bm{v}\cdot\tau d\Gamma -K\int_{\partial \Omega} q \nabla
p\cdot\bm{n} \, d\Gamma =0
\end{eqnarray*}
The first integral on the boundary $\partial \Omega$ vanishes since
$\bm{v}\in \mathcal{U}$. Using the fundamental lemma of the calculus of
variations \cite{GELFAND63}, and the identity
\begin{equation}
\label{eq:idrot}
\Delta \bm{u} = \nabla \Div \bm{u} - \curl \rot \bm{u}
\end{equation}
we have the Euler-Lagrange equations:
\begin{equation}
\label{eq:brin}
 \lambda \bm{u} -\lambda \Delta \bm{u} + \nabla p =
-\lambda \nabla f \quad \mbox{ in } \; \Omega,
\end{equation}
\begin{equation}
\label{eq:strg}
 \Div \bm{u} -K\Delta p = 2f  \quad \mbox{ in } \;  \Omega,
\end{equation}
{with boundary conditions \qquad\qquad\qquad}
\[
\rot \bm{u} = 0  \quad \mbox{ on } \;
\partial \Omega,
\]
\[
K\nabla p\cdot \bm{n}  = 0  \quad \mbox{ on } \;
\partial \Omega,
\]
\[
\bm{u} \cdot \bm{n}  = 0  \quad \mbox{ on } \;
\partial \Omega .
\]

To derive estimates in $[L^2(\Omega)]^2\times L^2(\Omega)$ norm for
velocity and potential, we use the following auxiliary problem: \
{\it Given $\{\bm{u}_h,p_h\}$ the solution of {\rm Problem CGLSh}, find
$\bm{w}\in\mathcal{U}\cap [H^2(\Omega)]^2$ and $\varphi \in \mathcal{P}\cap H^2(\Omega)$
such that}

\begin{eqnarray}
\qquad \qquad\qquad\qquad \label{eq:aux1} \lambda \bm{w} -\lambda
\Delta \bm{w} + \nabla \varphi  &=&
\bm{u}-\bm{u}_h \quad \mbox{ in } \; \Omega  \\
\label{eq:aux2} \Div \bm{w} -K\Delta \varphi &=& p-p_h  \quad \mbox{ in
} \;  \Omega  \\\nonumber \\
\rot \bm{w} &=& 0  \quad \mbox{ on } \;
\partial \Omega \\
K\nabla \varphi\cdot \bm{n}  &=& 0  \quad \mbox{ on } \;
\partial \Omega, \\
\bm{w} \cdot \bm{n}  &=& 0  \quad \mbox{ on } \;
\partial \Omega .
\end{eqnarray}

Multiplying Eq. (\ref{eq:aux1}) by $\bm{v}\in \mathcal{U}$, Eq. (\ref{eq:aux2})
by $q\in \mathcal{P}$ and integrating by parts we obtain
\[
B_{\mathrm{CGLS}}(\{\bm{w},\varphi\},\{\bm{v},q\}) = (\bm{u}-\bm{u}_h,\bm{v}) + (p-p_h,q).
\]
Taking $\{\bm{v},q\}=\{\bm{u}-\bm{u}_h,p-p_h\}$, using the orthogonality result and
the continuity of $ B_{\mathrm{CGLS}}(\cdot,\cdot)$
\begin{eqnarray}
\nonumber \|\bm{u}-\bm{u}_h\|^2 + \|p-p_h\|^2 &\leq& M \| \{\bm{w},\varphi\} - \{\bm{w}_h,\varphi_h\}\|_{\mathcal{U}\times\mathcal{P}}
\|\{\bm{u},p\}-\{\bm{u}_h,p_h\}\|_{\mathcal{U}\times\mathcal{P}} \\
&\leq&C h\left( \left| \bm{w}\right|_2 +
\left|\varphi\right|_2\right)\|\{\bm{u},p\}-\{\bm{u}_h,p_h\}\|_{\mathcal{U}\times\mathcal{P}} \label{eq:aux5}
\end{eqnarray}
To derive a bound for $\left( \left| \bm{w}\right|_2 +
\left|\varphi\right|_2\right)$, we multiply Eq. (\ref{eq:aux1}) by
$\Delta\bm{v}$, Eq. (\ref{eq:aux2}) by $-\Delta q$ and integrate by
parts to get, respectively
\begin{eqnarray}
\nonumber \lambda (\nabla \bm{w},\nabla \bm{v}) + \lambda (\Delta
\bm{w},\Delta\bm{v}) -(\nabla\varphi,\Delta\bm{v})-\lambda \int_{\partial
\Omega} (\nabla \bm{v})^T \bm{w}\cdot \bm{n} \, d \Gamma\\= (u-u_h,\Delta\bm{v})
\label{eq:aux3}
\end{eqnarray}
and
\begin{eqnarray}
\nonumber (\nabla q ,\Delta \bm{w}) + (\nabla q ,\curl \rot \bm{w}) +
\kappa(\Delta \varphi,\Delta q) - \int_{\partial \Omega} \Div \bm{w} \nabla
q \cdot \bm{n} \, d \Gamma\\ = (p-p_h,\Delta q). \label{eq:aux4}
\end{eqnarray}
Adding (\ref{eq:aux3}) to (\ref{eq:aux4}), taking $q=\varphi$,
$\bm{v}=\bm{w}$, using the boundary conditions and the fact that
$\bm{w}\in\mathcal{U}$, we get
\begin{eqnarray*}
\lambda \|\nabla \bm{w} \|^2 + \lambda \|\Delta\bm{w}\|^2 + K \|\Delta
\varphi\|^2 + (\nabla \varphi, \curl \rot \bm{w}) = (\bm{u}-\bm{u}_h,\Delta
\bm{w}) \\+ (p-p_h,\Delta \varphi).
\end{eqnarray*}
Observing that
\[
(\nabla \varphi,\curl\rot \bm{w}) = (\rot\nabla \varphi,\rot
\bm{w})+\int_{\partial \Omega}(\rot \bm{w}) \nabla \varphi \cdot \tau \,
d\Gamma = 0
\]
we have
\begin{eqnarray*}
\frac{1}{2}\min\{\lambda,K\}(  \|\nabla \bm{w} \| + \|\Delta\bm{w}\| &+&
\|\Delta \varphi \| )^2\\ \leq ( \|\bm{u}-\bm{u}_h\|&+&\|p-p_h\|)\left(
\|\nabla \bm{w}\|+\|\Delta\bm{w}\|+\|\Delta\varphi\|\right)
\end{eqnarray*}
and, since $ \left| \, \cdot \, \right|_{2}\leq C \|\Delta\cdot\|$,
the following inequality holds
\begin{equation}
   \left| \bm{w} \right|_1 + \left|\bm{w}\right|_2 + \left|\varphi
\right|_2 \leq C\left( \|\bm{u}-\bm{u}_h\| + \|p-p_h\| \right)
\label{eq:aux6}
\end{equation}
Combining (\ref{eq:aux6}) with (\ref{eq:aux5}), and using
(\ref{eq:cp-est}) we have the estimate for velocity and potential in
$[L^2(\Omega)]^2\times L^2(\Omega)$ norm:
\begin{equation}
\label{eq:optimal} \|\bm{u}-\bm{u}_h\| + \|p-p_h\| \leq C \left( h^{l+1}
\left| \bm{u} \right|_{l+1} + h^{k+1} \left| p \right|_{k+1} \right).
\end{equation}

\begin{remark}
To prove the optimal estimate {\rm (\ref{eq:optimal})}, the
$H^1(\Omega)^2\times H^1(\Omega)$ stability and the regularity
estimate {\rm (\ref{eq:aux6})} for velocity and potential are
essential. These properties come from the addition of least squares
residuals of the mass balance equation and of the curl of Darcy's
law to the dual mixed formulation.
\end{remark}

%
%
\section{Alternative Formulations}
\label{sec:otherfem}
%
%

%
%
%
%
%

The analysis of Problem CGLS indicates many other possibilities of
deriving unconditionally stable finite element methods for Darcy
flow by adequately including least square residuals. In this section
we comment some of these possibilities, concerning $C^0(\Omega)$
Lagrangian spaces.

%
%
\subsection{Symmetric Stabilizations in ${H(\div,\Omega)} \times H^1(\Omega)$}
%
%

Let $\tilde{\mathcal{U}}$ and $\mathcal{P}$ be the functional spaces for velocity and
potential, and their finite dimension subspaces
\begin{equation}
\label{eq:uh}
 \mathcal{U}_h^l = \left[ \mathcal{S}_h^l\right]^n \cap \tilde{\mathcal{U}}
\end{equation}
and $\mathcal{P}_h^{k}$, as defined in (\ref{eq:qh1}).
As we work with continuous Lagrangian interpolations,
the space (\ref{eq:uh}) is equivalent to that defined in
(\ref{eq:uh1}).

Considering only the least squares residuals of the mass balance and
of Darcy's law combined with the Dual mixed formulation
(\ref{eq:mpd}), we have the unconditionally stable formulation in
${H_0(\div,\Omega)}\times H^1(\Omega)/\mathbb{R}$:

\noindent{ \bf Problem GLS(Hdiv):} {\it Find $\{\bm{u},p\} \in \tilde{\mathcal{U}}\times\mathcal{P}$ such that, for
all $\{\bm{v},q\}\in \tilde{\mathcal{U}}\times\mathcal{P}$}
\begin{eqnarray}
\nonumber (\lambda\bm{u},\bm{v})- (\Div \bm{v}, p) &-& (\Div \bm{u}, q)  -
\frac{1}{2}\left( K\left( \lambda \bm{u} + \nabla p\right), \lambda \bm{v}
+ \nabla q \right) \\ &+& \frac{1}{2}\left(\lambda\Div \bm{u},\Div \bm{v}
\right) =- (f,q) + \frac{1}{2}(\lambda f, \Div\bm{v}).
\end{eqnarray}
This formulation preserves the saddle point structure of the Dual
mixed formulation. The analysis follows the same techniques used for
Problem CGLS leading to the error estimates:
\begin{equation}
    \label{eq:hdiv-p}
     \|p - p_h\|_1 \leq  C
     h^{k}\left( \left|\bm{u} \right|_{k+1}
     + \left|p \right|_{k+1} \right),
\end{equation}
\begin{equation}
    \label{eq:hdiv-u}
     \|\bm{u} - \bm{u}_h \| \leq  C  h^k\left( \left|\bm{u} \right|_{k+1}
     + \left|p \right|_{k+1} \right),
\end{equation}
\begin{equation}
    \label{eq:hdiv-div}
     \|\Div\bm{u} - \Div\bm{u}_h \| \leq  C  h^k\left( \left|\bm{u} \right|_{k+1}
     + \left|p \right|_{k+1} \right),
\end{equation}
whose rates of convergence are optimal for potential in
$H^1(\Omega)$ norm and for the divergence of the velocity in
$L^2(\Omega)$ norm but suboptimal for the velocity in
$[L^2(\Omega)]^n$ norm.
This method should not be seen as just a particular case of the
method CGLS. It is another method in the sense that its ${H(\div,\Omega)}$
stability allows the use of divergent based finite element spaces
for velocity, as those developed in \cite{RAVIART77}. This
possibility, although not explored in the present work, is of great
interest in the study of heterogeneous porous media with interfaces
of discontinuity of the hydraulic resistivity.

A similar method was proposed in the eighties \cite{LOULA88} for a
heat transfer problem with the same mathematical structure of Darcy
problem. Differently of GLS(Hdiv), in \cite{LOULA88} the authors
consider both least squares residuals multiplied by positive
constants, resulting in the following method: {\it Find $\{\bm{u}_h,p_h\} \in
\mathcal{U}_h^{l}\times\mathcal{P}_h^{k}$ such that for all $\{\bm{v}_h,q_h\} \in \mathcal{U}_h^{l} \times \mathcal{P}_h^{k}$
\begin{eqnarray}
\nonumber (\lambda\bm{u},\bm{v})- (\Div \bm{v}, p) &-& (\Div \bm{u}, q)  +
\delta_1\left( K\left( \lambda \bm{u} + \nabla p\right), \lambda \bm{v}
+ \nabla q \right) \\ &+& \delta_2\left(\lambda\Div \bm{u},\Div \bm{v}
\right) =- (f,q) + \delta_2(\lambda f, \Div\bm{v}).
\end{eqnarray}
}

For appropriate choices of $\delta_1>0$ and  $\delta_2>0$, this
method is equivalent to a minimization problem, which assures
existence, uniqueness and the best approximation property in the
associated energy norm. The numerical analysis of this method, that
we call MGLS ({\em Minimum}/Galerkin Least-Squares), leads the same
error estimates of GLS(Hdiv).

%
%
\subsection{Symmetric Stabilization in $\left[ L^2(\Omega)\right]^n
\times H^1(\Omega)$}
%
%

A non symmetric stabilization  for Darcy flow is proposed in
\cite{MASUD2002} through the addition of an adjoint residual of the
Darcy's law to the non symmetric form of the Dual mixed formulation,
yielding the following method: {\it Find $\{\bm{u}_h,p_h\} \in \mathcal{U}_h^{l}\times\mathcal{P}_h^{k}$
such that for all $\{\bm{v}_h,q_h\} \in \mathcal{U}_h^{l} \times \mathcal{P}_h^{k}$
 \begin{equation}
 (\lambda\bm{u},\bm{v})- (\Div \bm{v}, p) + (\Div \bm{u}, q) \label{eq:hvm}
+ \frac{1}{2}\left( K\left( \lambda \bm{u} + \nabla p\right), - \lambda
\bm{v} + \nabla q \right)=(f,q).
  \end{equation}
}

In \cite{NAKSHATRALA2006} it is shown that this formulation can be
derived from a variational multiscale approach and it is referred as
Hughes variational multiscale (HVM) formulation for Darcy flow.
For sufficiently regular exact solutions and same order
interpolations ($l=k$), we derive the error estimates:
\begin{equation}
    \label{eq:hvm-p}
     \| \nabla p - \nabla p_h  \|\leq  C
     h^{k}\left( \left|\bm{u} \right|_{k+1}
     + \left|p \right|_{k+1} \right),
\end{equation}
\begin{equation}
    \label{eq:hvm-u}
     \|\bm{u} - \bm{u}_h \| \leq  C  h^k\left( \left|\bm{u} \right|_{k+1}
     + \left|p \right|_{k+1} \right),
\end{equation}
with  optimal rates for potential in  $H^1(\Omega)$ seminorm
(equivalent to $H^1(\Omega)$ norm) but suboptimal for the velocity
in $[L^2(\Omega)]^n$ norm.

In \cite{MASUD2002} the authors say that stabilizations such as
Galerkin/Least Squares are not as effective for the current problem
as the adjoint HVM formulation. The only essential difference is the
sign on the $\bm{v}$ term which they considered crucial. In fact due to
the non symmetry one can easily prove stability of HVM in
$[L^2(\Omega)]^n$ norm for the velocity and $H^1(\Omega)$ seminorm
for the potential, in the sense
of Lax.
However, taking $q=-q$ in (\ref{eq:hvm}) and integrating by parts,
we can present an equivalent symmetric formulation, posed in
$\bar{\mathcal{U}}=\left[ L^2(\Omega) \right]^n$ and $\mathcal{P}=H^1(\Omega)/ \mathbb{R}$:

\noindent{ \bf Problem SHVM:} {\it Find $\{\bm{u},p\} \in \bar{\mathcal{U}}\times\mathcal{P}$ such that for
all $\{\bm{v},q\} \in \bar{\mathcal{U}}\times\mathcal{P}$
  \begin{equation}
    (\lambda \bm{u},\bm{v}) + (\bm{v},\nabla p) + (\bm{u},\nabla q)
    -\frac{1}{2}\left( K\left( \lambda \bm{u} + \nabla p\right),
    \lambda \bm{v} + \nabla q\right)= -(f,q).
    \label{eq:primalcs}
  \end{equation}
}
This symmetric formulation is in fact a combination of the least
square residual of Darcy's law with the Primal Mixed Formulation,
described in Section \ref{sec:pmf}.  Although this formulation is
not stable in the sense of Lax, like the HVM, it is unconditionally
stable in the sense of Bab\v uska. The analysis follows from that
described for Problem CGLS.
Thus, taking only the least square residual of Darcy's law in our
CGLS formulation (\ref{eq:g-bilinear}), we have a symmetric and
unconditionally stable method, equivalent to the adjoint HVM
formulation.

Numerical results  presented in \cite{LOULA2006A} show that both
methods have the same accuracy. In fact, as we show in the Section
\ref{sec:adit}, the HVM formulation is identical to this symmetric
method (SHVM). Again, we emphasize that this method is not a
particular case of the CGLS or GLS(Hdiv) formulations, since
divergent based as well as continuous or discontinuous Lagrangian
finite element spaces can be adopted for the velocity field,
associated with continuous Lagrangian interpolations for the
potential.

\begin{remark}
The {\rm GLS(Hdiv)} method presents optimal estimate for velocity in
the ${H(\div,\Omega)}$ norm, but not necessarily in $[L^2(\Omega)]^2$ norm.
However, optimal estimates for potential can be proved for this
method, as well as for the {\rm SHVM} method.
\end{remark}

%
%
\subsection{A Stokes Compatible Formulation}
%
%

Excluding the least squares residual of Darcy's law in the {\rm
CGLSh } formulation presented above, we derive the Stokes compatible
finite element stabilization for Darcy's problem with velocity in
$\left[H^1(\Omega)\right]^n$ and pressure in $L^2(\Omega)$: {\it
Find $\{\bm{u}_h,p_h\} \in \mathcal{U}_h^{l} \times \mathcal{P}_h^{k}$ such that
\begin{eqnarray}
 a(\bm{u}_h,\bm{v}_h) + b(\bm{v}_h,p_h) &=& g(\bm{v}_h) \quad \forall \; \bm{v}_h \in
 \mathcal{U}_h^{l} \label{eq:scomp1}\\
 b(\bm{u}_h,q_h) &=& f(q_h) \quad \forall \; q_h \in \mathcal{P}_h^{k} \label{eq:scomp2}
\end{eqnarray}
} with
\[
 a(\bm{u},\bm{v}) = (\lambda\bm{u},\bm{v})
 + \frac{1}{2}\left(\lambda\Div \bm{u},\Div \bm{v} \right)
 +\frac{1}{2}\left(K\curl(\lambda \bm{u}),\curl (\lambda \bm{v}) \right)
\]
\[
 b(\bm{v},q) =  - (\Div \bm{v}, q)
\]
\[
 g(\bm{v}) = \frac{1}{2}\left(f,\Div \bm{v} \right)\, ;  \ \ f(q) = (f,q)
 \, .
\]
Clearly, this formulation is not unconditionally stable. Given the
equivalence of norms stated in Lemma \ref{lem:eq}, the similarity
between this formulation and the Galerkin formulation of Stokes
problem becomes evident. Thus, all  Galerkin stable finite element
spaces for Stokes are also stable for this formulation. The choice
$l=k+1$ with continuous velocity and pressure,  the well known
Taylor-Hood elements \cite{VERFURTH84}, leads to optimal rates of
convergence for velocity in $\left[H^1(\Omega)\right]^n$ and
pressure in $L^2(\Omega)$. This type of compatible formulation may
play an important role in the approximation of Stokes-Darcy coupled
flows \cite{CORREA2006}. On finite element approximations of
Stokes-Darcy coupled problem see references
\cite{DISCACCIATI2002,LAYTON2003,RIVIERE2005B,ARBOGAST2007}.

%
%
\subsection{Additional Comments}
\label{sec:adit}
%
%

An interesting aspect of the former stabilization may be viewed by
rearranging (\ref{eq:primalcs}), leading to the variational
formulation: {\it Find $\{\bm{u},p\} \in \bar{\mathcal{U}}\times\mathcal{P}$ such that for all $\{\bm{v},q\}\in
\bar{\mathcal{U}}\times\mathcal{P}$}
  \begin{equation}
     (\lambda \bm{u},\bm{v})  + (\bm{v},\nabla p) + (\bm{u},\nabla q) -
     \left(K\nabla p,\nabla q \right)= - 2(q,f). \label{eq:h1param}
  \end{equation}
We should note that this formulation is obtained by subtracting the
Primal formulation {\rm (\ref{eq:pv})} from the Primal mixed
formulation (Problem PM). This equivalence is clearly observed
rewriting {\rm (\ref{eq:h1param})} as: {\it Find $\bm{u} \in \bar{\mathcal{U}}$ and
$p \in \mathcal{P}$ such that}%
\begin{eqnarray}
\label{eq:equil1}
     (\lambda \bm{u},\bm{v})  &+& (\bm{v},\nabla p) =0  \;\; \quad \qquad \forall \; \bm{v} \in \bar{\mathcal{U}}
     \\
\label{eq:equil2}
      -\left[\left(K\nabla p,\nabla q \right)-  (f,q)\right] &+&(\bm{u}, \nabla q)=
      -(f,q) \quad \forall \; q \in \mathcal{P}.
\end{eqnarray}

The above equations show that an unconditionally stable mixed method
is derived  combining two classical Galerkin formulations.
We also observe that multiplying equation (\ref{eq:equil2}) by $-1$,
combining with equation (\ref{eq:equil1}) and integrating by parts,
we recover the HVM formulation. Thus, HVM and SHVM methods are
identical for $C^0(\Omega)$ Lagrangian elements, and both can be
derived from the same variational formulation.

Similarly, the CGLS formulation can be obtained by the addition of
the least squares residuals of the mass balance equation and the
curl of Darcy' s law to the equation (\ref{eq:equil1}), yielding the
following method: {\it Find $\bm{u} \in \mathcal{U}$ and $p \in \mathcal{P}$ such that
for all $\bm{v} \in \mathcal{U}$ and $q \in \mathcal{P}$ }
\begin{eqnarray*}
     (\lambda \bm{u},\bm{v}) + \left(\lambda\Div \bm{u},\Div \bm{v} \right) + \left(K\curl(\lambda \bm{u}),\curl
(\lambda \bm{v}) \right) + (\bm{v},\nabla p)  &=& (\lambda f, \Div\bm{v})
 \\
      -\left[\left(K\nabla p,\nabla q \right)-  (f,q)\right] +(\bm{u},\nabla q)
      &=&
      -(f,q).
\end{eqnarray*}

Concluding this Subsection, we recall a nonsymmetric formulation
presented in \cite{CORREA2007A}:  {\it Find $\{\bm{u},p\} \in \mathcal{U}\times\mathcal{P}$ such that,
for all $\{\bm{v},q\}\in \mathcal{U}\times\mathcal{P}$
\begin{eqnarray}
 \nonumber (\lambda\bm{u},\bm{v}) - (\Div \bm{v}, p) &+& (\Div \bm{u}, q) \\
  \nonumber &+& \delta_1\left( K\left( \lambda \bm{u} + \nabla p\right), \lambda
\bm{v} + \nabla q  \right) \\ \nonumber & +& \delta_2\left(\lambda\Div
\bm{u},\Div \bm{v} \right) +\delta_3\left(K\curl(\lambda \bm{u}),\curl
(\lambda \bm{v}) \right) \nonumber \\
  &=& (f,q) +
\delta_2(\lambda f, \Div\bm{v})
\end{eqnarray}
}
where $\delta_i$ are free positive parameters. This nonsymmetric
formulation is inspired on the stabilization for the Stokes problem
proposed in \cite{DOUGLAS89}. It leads to the same convergence rates
obtained for CGLS. These facts are discussed and numerically
illustrated in \cite{CORREA2007A}.

%
%
\section{Numerical Experiments} \label{sec:num}
%
%

In this Section we present two numerical studies to confirm the
predicted rates of convergence and to demonstrate the flexibility
and the robustness of the proposed methods.

%
%
\subsection{Convergence Study}
%
%

In this first numerical experiment we perform a  convergence study
of CGLS, HVM, GLS(Hdiv) and MGLS (with $\delta_1=\delta_2=1/2$)
methods, applied to homogeneous and non-homogeneous media with
regular solutions. In this study we solve Darcy problem in a
bidimensional square domain $[0,2]\times[0,2]$ with regular
conductivity
\begin{equation}
\label{eq:k} \kappa(x,y)=k_1(x-2)x(y-2)y+k_2
\end{equation}
with $k_1\geq0$ and $k_2>0$, and source
\begin{eqnarray*}
f(x,y)=&+ & [k_1(x-2)x(y-2)y+k_2]\sin \pi x \sin \pi y
\\ &-&\frac{k_1}{\pi}\left[ (x-1)(y-2)y\cos \pi x \sin \pi x
\right. \\ && \qquad +\left. (x-2)x(y-1)\sin \pi x \cos \pi y
\right],
\end{eqnarray*}
such that the exact solution is given by
\begin{equation}
  \label{eq:pe-s}
  p= \frac{1}{2\pi^2}\sin{\pi x}\sin {\pi y},
\end{equation}
\begin{equation}
  \label{eq:ue-s}
  \bm{u} = - \frac{[k_1(x-2)x(y-2)y+k_2]}{2\pi}
  \left[
    \begin{array}{c}
      \cos{\pi x}\sin {\pi y}  \\
      \sin{\pi x}\cos {\pi y}
    \end{array}
    \right].
\end{equation}
The flux is prescribed on the boundary using (\ref{eq:ue-s}).
The finite element solutions were computed adopting equal order
interpolation for velocity and pressure and uniform meshes with
$8\times8$, $16\times16$, $32\times32$ and $64\times64$ bilinear
elements and of $4\times4$, $8\times8$, $16\times16$ and
$32\times32$ biquadratic and bicubic elements for both fields.
Numerical integrations were performed with $3\times 3$, $4\times 4$
and $5\times5$ Gauss quadrature for Q$_{1}$Q$_{1}$, Q$_{2}$Q$_{2}$
and Q$_{3}$Q$_{3}$ elements, respectively.

%
%
\subsubsection{Homogeneous Medium}
%
%

In this study we consider  $k_1=0.0$ and $k_2=1.0$ in (\ref{eq:k}),
reproducing the example studied in
\cite{MASUD2002,LOULA2006A,CORREA2006}. To illustrate the solution
of  the test example in this case, we show in Figure \ref{fig:ph}
the potential and in Figure \ref{fig:vxh} the $u_x$ component of the
velocity field, approximated by CGLS with $32\times 32$ bilinear
elements (the  $u_y$ component is similar to $u_x$).

\begin{figure}[htb]
\centering
\subfloat[Potential $p$.]{\includegraphics[angle=0,width=.49\textwidth,angle=0]{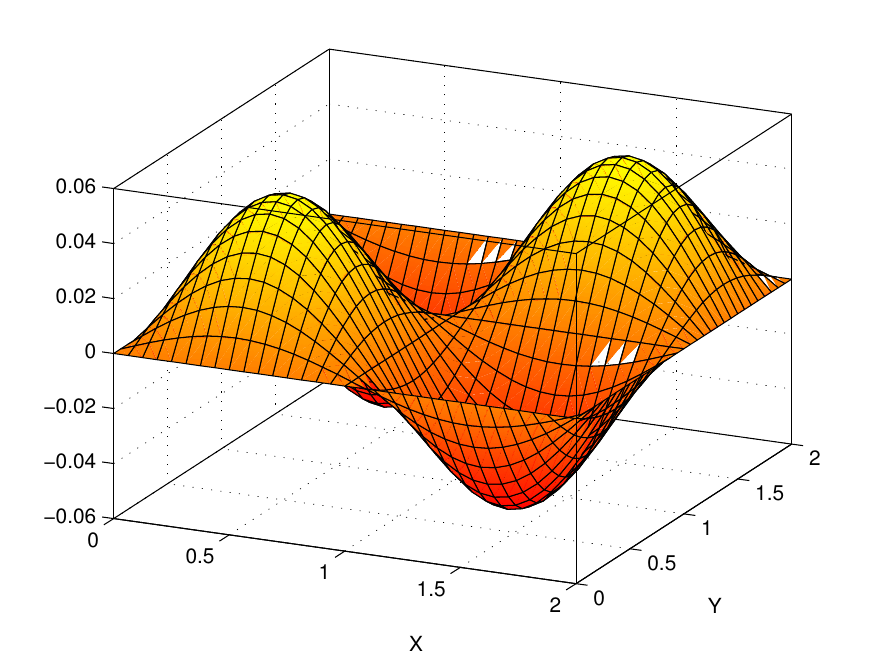}
\label{fig:ph}}
%
%
\subfloat[$u_x$-component of the velocity.]{\includegraphics[angle=0,width=.49\textwidth,angle=0]{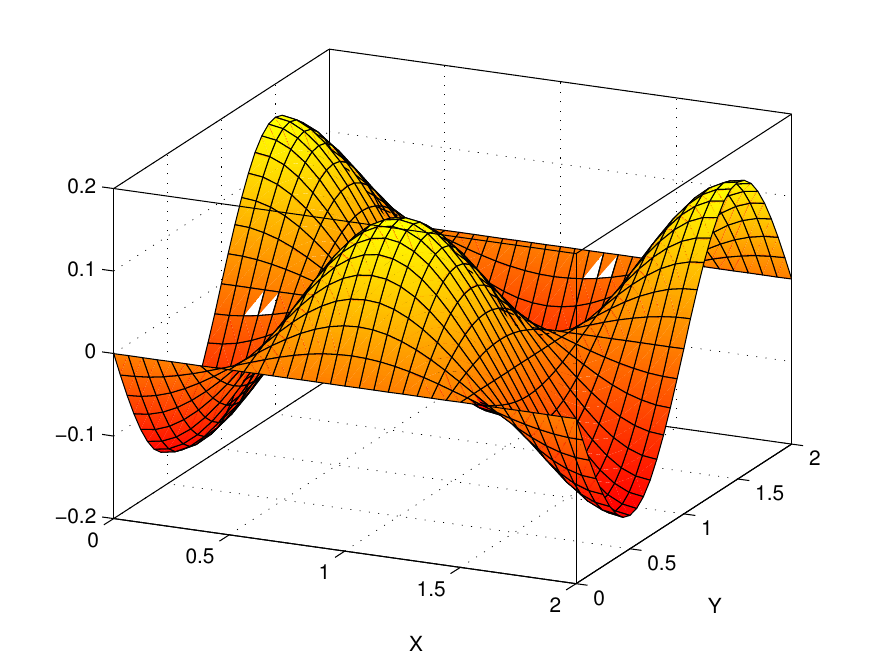}
\label{fig:vxh}}
\caption{Homogeneous Medium. Approximate Solution for $k_1=0.0$ and $k_2=1.0$}
\label{fig:graph1}
\end{figure}

In the convergence study, we present plots of the $\log$ of
$L^2(\Omega)$ norm of the errors of both velocity and potential,
their gradients and the divergence of the velocity {\em versus}
$-\log(h)$.
The results for velocity using Q$_{1}$Q$_{1}$ elements are shown in
Figure \ref{fig:vel-seno-q1}. Convergence rates close to
$O(h^{1.5})$ are obtained for GLS(Hdiv) and MGLS, while the optimal
convergence rate $O(h^{2.0})$ is achieved by CGLS and HVM.
According to numerical analysis, we may expect the optimal rate
$O(h^{2.0})$ only for CGLS. The convergence rates observed are
higher in one order than that predict for HVM and in half order for
GLS(Hdiv) and MGLS.
\begin{figure}[htb]
\centering
\subfloat[Velocity.]{\includegraphics[angle=0,width=.49\textwidth,angle=0]{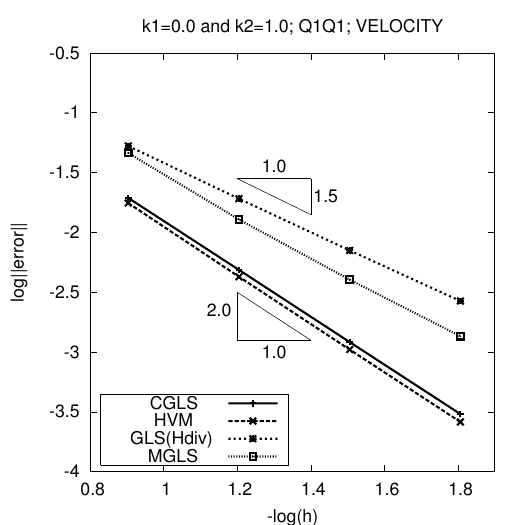}
\label{fig:vel-seno-q1}}
%
%
\subfloat[Gradient of velocity.]{\includegraphics[angle=0,width=.49\textwidth,angle=0]{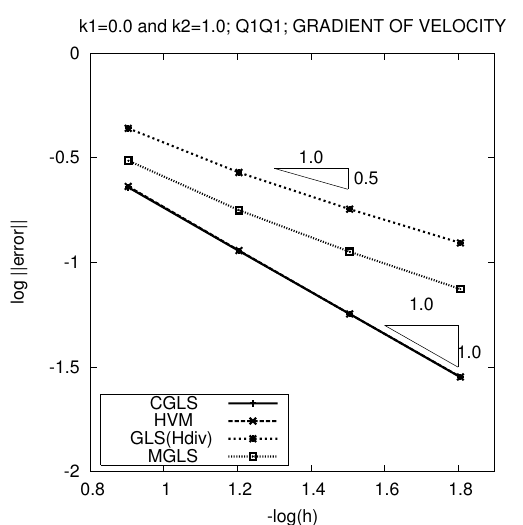}
\label{fig:grad-vel-seno-q1}}
\caption{Homogeneous Medium. 
 Convergence study: Q$_{1}$Q$_{1}$ elements.}
\label{fig:v-seno-q1}
\end{figure}
The rates for the gradient of the velocity, shown in the plots of
Figure \ref{fig:grad-vel-seno-q1}, also indicate convergence order
higher than predicted, once the rate  $O(h^{1.0})$ in
$[H^1(\Omega)]^2$ seminorm is reached for HVM, although expected
only for CGLS.

This ``super'' convergence is not observed for Q$_{2}$Q$_{2}$ and
Q$_{3}$Q$_{3}$ elements, as can be seen in Figures
\ref{fig:v-seno-q2} and \ref{fig:v-seno-q3}, respectively:
CGLS yields optimal  convergence rates for velocity in
$[H^1(\Omega)]^2$ seminorm and in $[L^2(\Omega)]^2$ norm, while the
expected quasi optimal rates are attained for the other methods.
\begin{figure}[htb]
\centering
\subfloat[Velocity.]{\includegraphics[angle=0,width=.49\textwidth,angle=0]{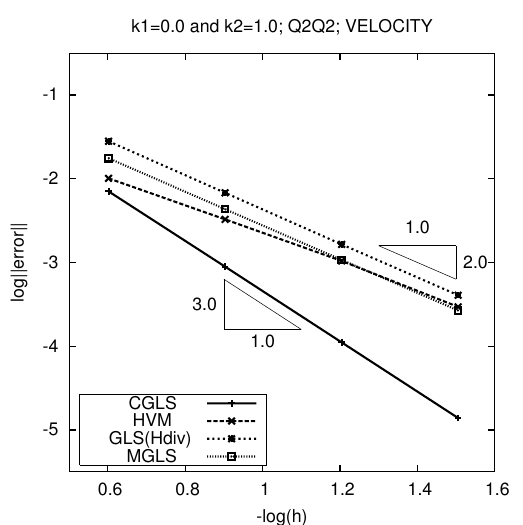}
\label{fig:vel-seno-q2}}
%
%
\subfloat[Gradient of velocity.]{\includegraphics[angle=0,width=.49\textwidth,angle=0]{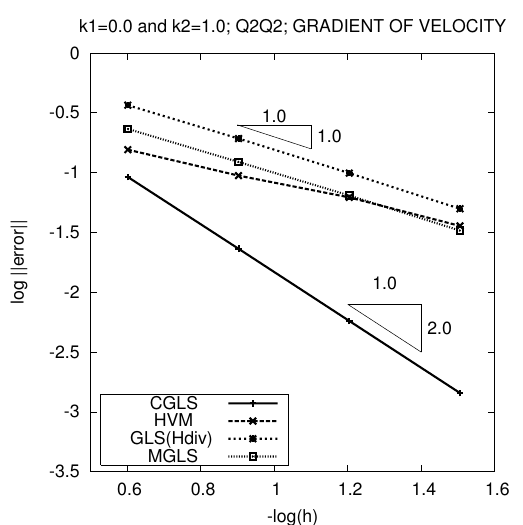}
\label{fig:grad-vel-seno-q2}}
\caption{Homogeneous Medium. Convergence study: Q$_{2}$Q$_{2}$ elements.}
\label{fig:v-seno-q2}
\end{figure}
\begin{figure}[htb]
\centering
\subfloat[Velocity.]{\includegraphics[angle=0,width=.49\textwidth,angle=0]{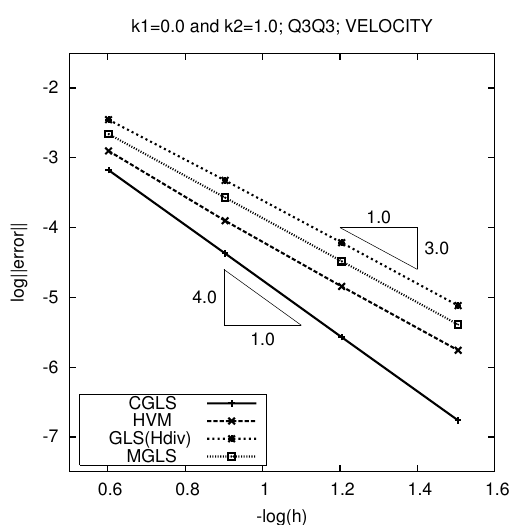}
\label{fig:vel-seno-q3}}
%
%
\subfloat[Gradient of velocity.]{\includegraphics[angle=0,width=.49\textwidth,angle=0]{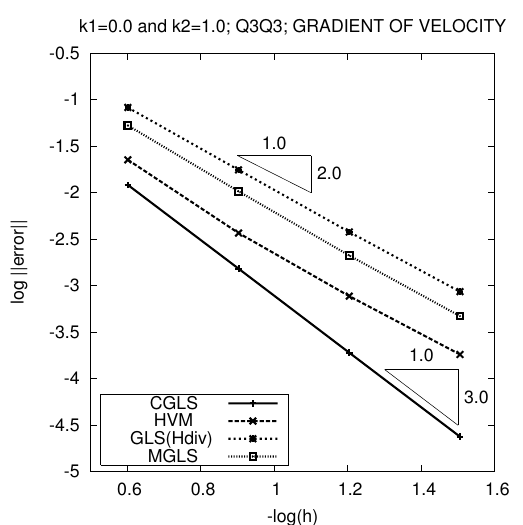}
\label{fig:grad-vel-seno-q3}}
\caption{Homogeneous Medium. Convergence study: Q$_{3}$Q$_{3}$ elements.}
\label{fig:v-seno-q3}
\end{figure}
The results for potential confirm the optimal convergence rates in
$H^1(\Omega)$ seminorm, leading additionally to optimal rates in
$L^2(\Omega)$ norm for all methods. These convergence plots are
shown in Figures \ref{fig:p-seno-q1}, \ref{fig:p-seno-q2} and
\ref{fig:p-seno-q3} for the Q$_1$Q$_1$, Q$_2$Q$_2$ and Q$_3$Q$_3$
elements, respectively.
\begin{figure}[htb]
\centering
\subfloat[Potential.]{\includegraphics[angle=0,width=.49\textwidth,angle=0]{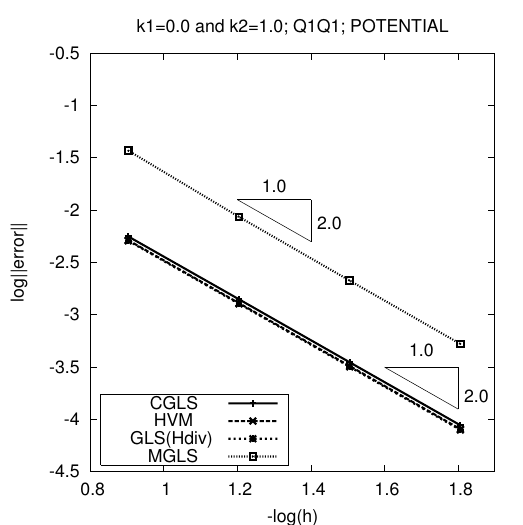}
\label{fig:pot-seno-q1}}
%
%
\subfloat[Gradient of potential.]{\includegraphics[angle=0,width=.49\textwidth,angle=0]{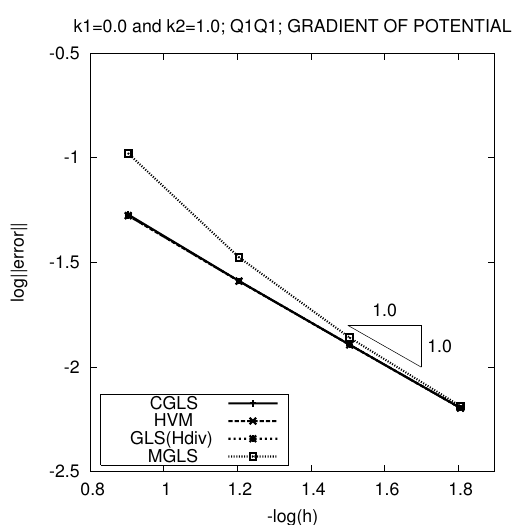}
\label{fig:grad-pot-seno-q1}}
\caption{Homogeneous Medium. Convergence study: Q$_{1}$Q$_{1}$ elements.}
\label{fig:p-seno-q1}
\end{figure}
\begin{figure}[htb]
\centering
\subfloat[Potential.]{\includegraphics[angle=0,width=.49\textwidth,angle=0]{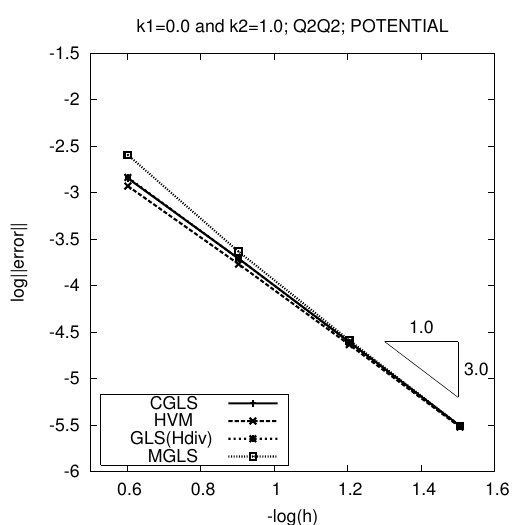}
\label{fig:pot-seno-q2}}
%
%
\subfloat[Gradient of potential.]{\includegraphics[angle=0,width=.49\textwidth,angle=0]{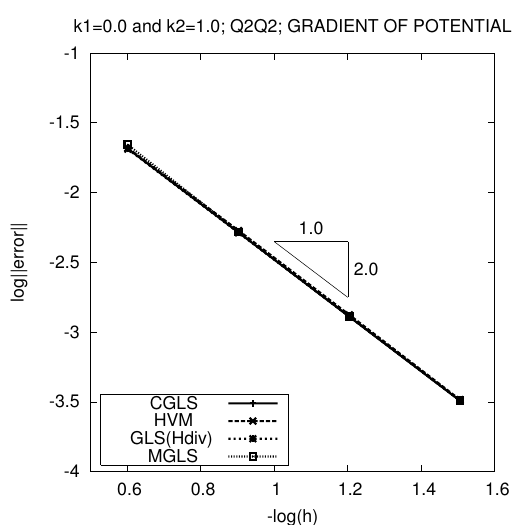}
\label{fig:grad-pot-seno-q2}}
\caption{Homogeneous Medium. Convergence study: Q$_{2}$Q$_{2}$ elements.}
\label{fig:p-seno-q2}
\end{figure}
\begin{figure}[htb]
\centering
\subfloat[Potential.]{\includegraphics[angle=0,width=.49\textwidth,angle=0]{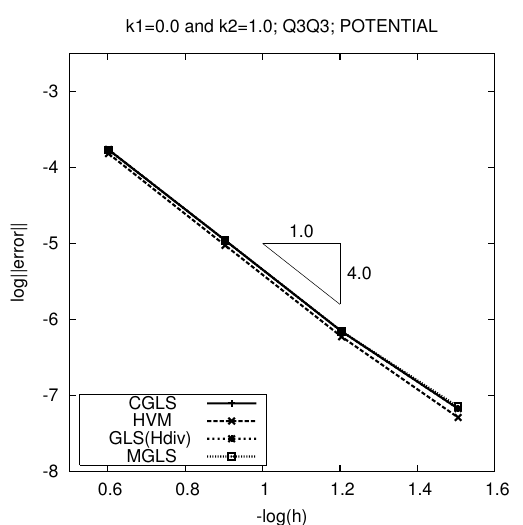}
\label{fig:pot-seno-q3}}
%
%
\subfloat[Gradient of potential.]{\includegraphics[angle=0,width=.49\textwidth,angle=0]{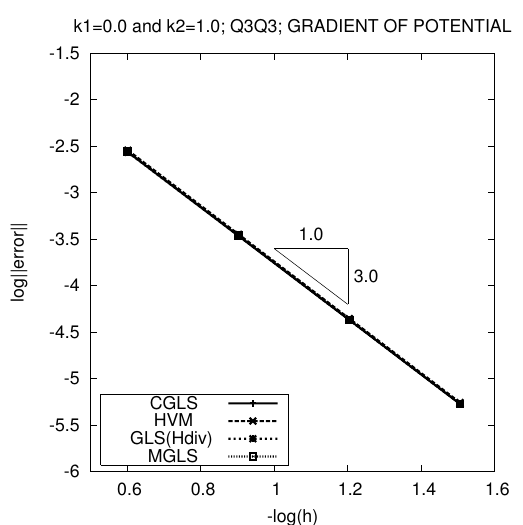}
\label{fig:grad-pot-seno-q3}}
\caption{Homogeneous Medium. Convergence study: Q$_{3}$Q$_{3}$ elements.}
\label{fig:p-seno-q3}
\end{figure}
In Figure {\rm\ref{fig:v-seno-q1}}, we observe that the addition of
least squares residual of the mass balance equation
--- used by  {\rm GLS(Hdiv)} and {\rm MGLS} --- apparently deteriorates the
convergence rates for velocity in $[L^2(\Omega)]^2$ norm for {\rm
Q$_{1}$Q$_{1}$} elements. This fact may give the impression that
methods that combine only residuals of Darcy's law (as {\rm HVM} or
its symmetric equivalent) are more accurate than GLS(Hdiv) and MGLS.
However, the numerical analysis shows that the  $[L^2(\Omega)]^2$
stability given by the combination of the adjoint or least squares
residual of Darcy's law is extended to  ${H(\div,\Omega)}$ by the addition of
the least square residual of mass balance improving stability
properties.
To illustrate this, we show in Figure {\rm \ref{fig:div-seno-th}}
the error plots of $\|\Div \bm{u} - \Div \bm{u}_h\|$ for {\rm Q$_2$Q$_2$}
and {\rm Q$_3$Q$_3$} elements. Optimal convergence rates are
achieved by  {\rm CGLS, GLS(Hdiv)} and {\rm MGLS}, while the
predicted suboptimal rates are observed for HVM.
Finally, the numerical experiments showed optimal convergence rates
for potential  in the $H^1(\Omega)$ seminorm or in the $L^2(\Omega)$
norm for all methods (not presented here).
\begin{figure}[htb]
\centering
\subfloat[Q$_{2}$Q$_{2}$ elements.]{\includegraphics[angle=0,width=.49\textwidth,angle=0]{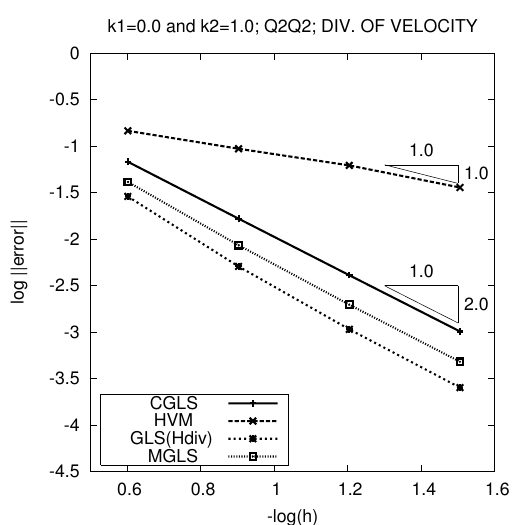}
\label{fig:div-vel-seno-q2}}
%
%
\subfloat[Q$_{3}$Q$_{3}$ elements.]{\includegraphics[angle=0,width=.49\textwidth,angle=0]{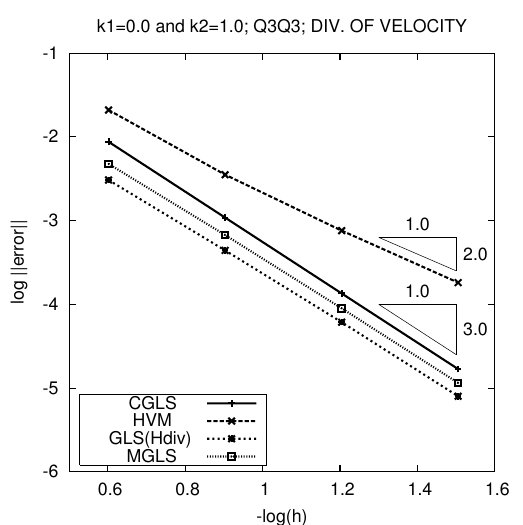}
\label{fig:div-vel-seno-q3}}
\caption{Homogeneous Medium. Convergence study: divergence of velocity. }
\label{fig:div-seno-th}
\end{figure}
%

%
%
\subsubsection{Non-Homogeneous Medium}
%
%

To investigate the stability of the CGLS method we check its
convergence properties when the conductivity is given by
(\ref{eq:k}) with $k_1=1.0$, $k_2=1.0$ and with $k_1=10.0$,
$k_2=1.0$.  The
 $u_x$ solutions approximated with a mesh of $32\times32$
 elements Q$_1$Q$_1$ are shown in Figure \ref{fig:graph2}.
\begin{figure}[htb]
\centering
\subfloat[$k_1=1.0$.]{\includegraphics[angle=0,width=.49\textwidth,angle=0]{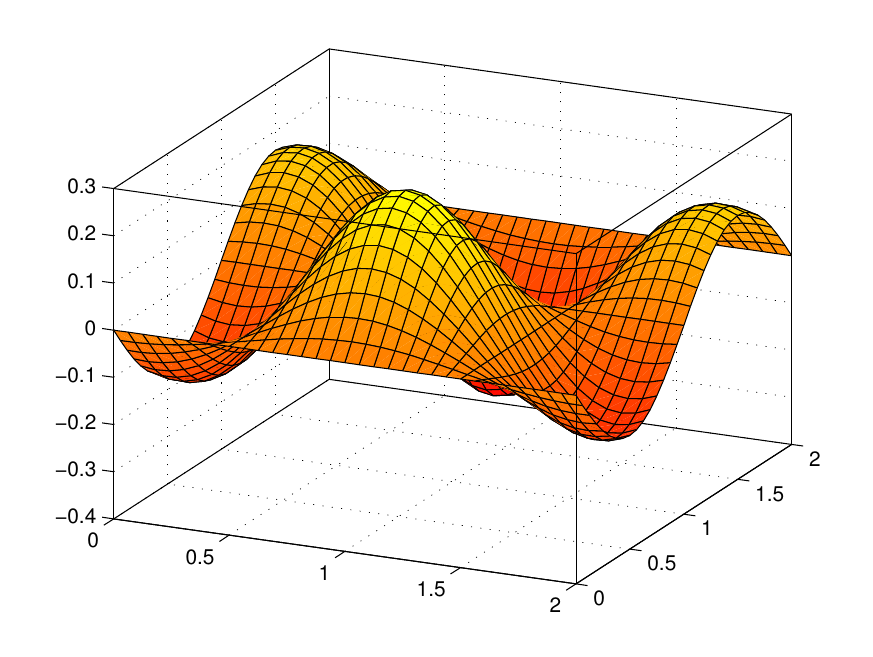}
\label{fig:vx1}}
%
%
\subfloat[$k_1=10.0$.]{\includegraphics[angle=0,width=.49\textwidth,angle=0]{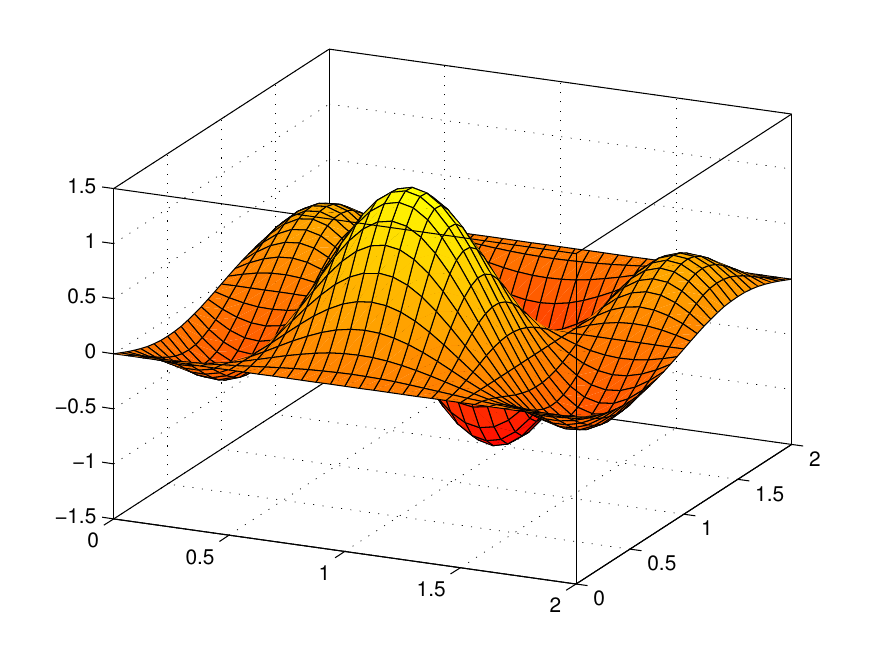}
\label{fig:vx10}}
\caption{Non-Homogeneous Medium. 
 Approximate $u_x$ component of the velocity for $k_2=1.0$.}
\label{fig:graph2}
\end{figure}
The results for the velocity are shown in Figures \ref{fig:gls-q1},
\ref{fig:gls-q2} and \ref{fig:gls-q3} for Q$_{1}$Q$_{1}$,
Q$_{2}$Q$_{2}$ and Q$_{3}$Q$_{3}$ elements, respectively. For the
sake of comparison the results of the first study, with constant
conductivity, are also plotted. As we can see, in all cases the
stability of the method CGLS remains unchanged and optimal
convergence rates, even in the $\left[L^2(\Omega)\right]$ norm, are
always achieved.
\begin{figure}[htb]
\centering
\subfloat[Velocity.]{\includegraphics[angle=0,width=.49\textwidth,angle=0]{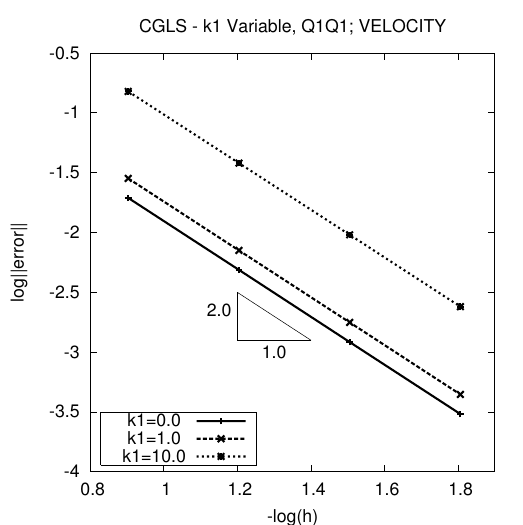}
\label{fig:gls-vel-q1}}
%
%
\subfloat[Gradient of velocity.]{\includegraphics[angle=0,width=.49\textwidth,angle=0]{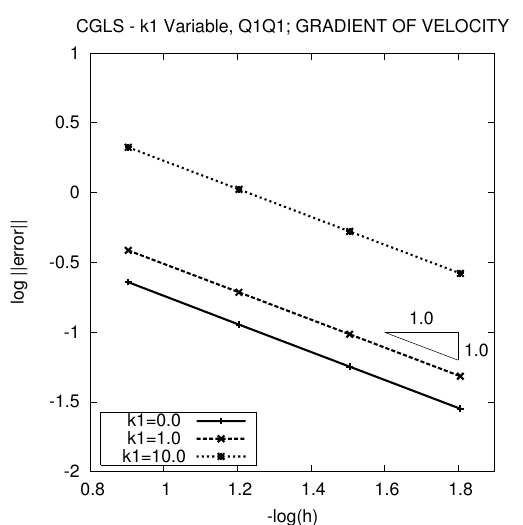}
\label{fig:gls-grad-vel-q1}}
\caption{Non-Homogeneous Medium. Convergence study for CGLS with variable conductivity
and  Q$_{1}$Q$_{1}$ elements.}
\label{fig:gls-q1}
\end{figure}
\begin{figure}[htb]
\centering
\subfloat[Velocity.]{\includegraphics[angle=0,width=.49\textwidth,angle=0]{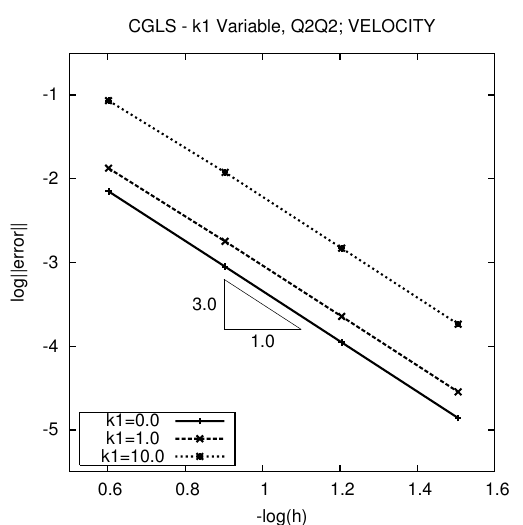}
\label{fig:gls-vel-q2}}
%
%
\subfloat[Gradient of velocity.]{\includegraphics[angle=0,width=.49\textwidth,angle=0]{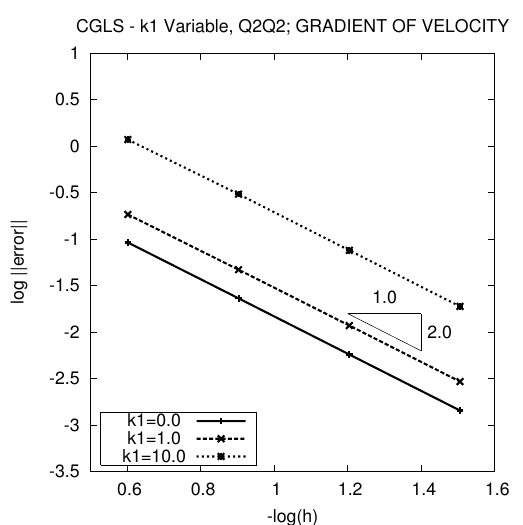}
\label{fig:gls-grad-vel-q2}}
\caption{Non-Homogeneous Medium. Convergence study for CGLS with variable conductivity
and  Q$_{2}$Q$_{2}$ elements.}
\label{fig:gls-q2}
\end{figure}
\begin{figure}[htb]
\centering
\subfloat[Velocity.]{\includegraphics[angle=0,width=.49\textwidth,angle=0]{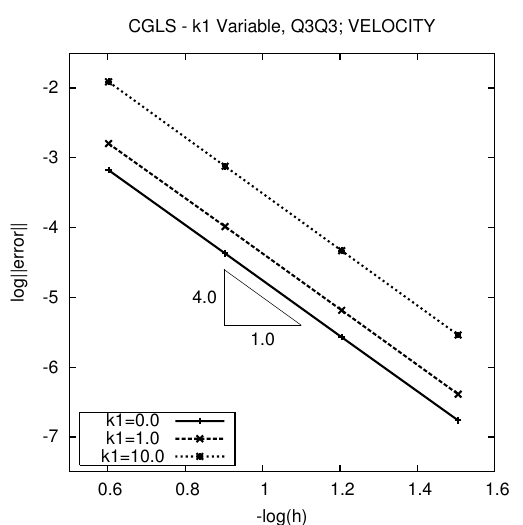}
\label{fig:gls-vel-q3}}
%
%
\subfloat[Gradient of velocity.]{\includegraphics[angle=0,width=.49\textwidth,angle=0]{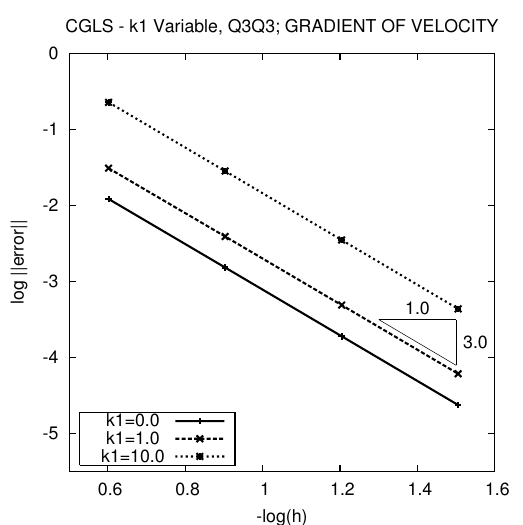}
\label{fig:gls-grad-vel-q3}}
\caption{Non-Homogeneous Medium. Convergence study for CGLS with variable conductivity
and  Q$_{3}$Q$_{3}$ elements.}
\label{fig:gls-q3}
\end{figure}
The errors for the potential and its gradient remains almost the
same.

\clearpage

%
%
\subsection{The Five-Spot Problem}
%
%

In this study, we examine the application of the proposed methods to
a problem with non regular solution corresponding to a quarter of
the five-spot problem. The mathematical model consists in solving
Darcy problem in a homogeneous medium ($k=1$) represented by a
square domain $\Omega=[0,L]\times[0,L]$ $(L=1.854075)$ with
concentrated loads: a sink at $x=y=0.0$ and a source at $x=y=L$,
with strength -1/4 and +1/4, respectively. Due the symmetry of the
problem, zero normal flow is prescribed on the boundary. The exact
solution for potential is given by  \cite{MOREL65}
\begin{equation}
  \label{eq:solfive}
  p(x,y) = \frac{1}{4\pi}\ln
  \left[ \frac{1-\mathrm{cn}^2(x)\mathrm{cn}^2(y)}
    {\mathrm{cn}^2(x)\mathrm{cn}^2(y)}\right],
\end{equation}
where $\mathrm{cn}$ denotes the elliptic cosine with modulus
$1/\sqrt{2}$.

To approximate this problem we first introduce a regularization.
Since we assumed $f\in L^2(\Omega)$, the concentrated loads are
approximated by the uniformly distributed loads over the corner
elements
\[
f_{\mathrm{sink}}=-\frac{0.25}{|\Omega^e_{\mathrm{sink}}|} \quad
\mbox{in} \quad \Omega^e_{\mathrm{sink}} \qquad \mbox{and} \qquad
f_{\mathrm{source}}=\frac{0.25}{|\Omega^e_{\mathrm{source}}|} \quad
\mbox{in} \quad \Omega^e_{\mathrm{source}}
\]
where  $|\Omega^e_{\mathrm{sink}}|$ and
$|\Omega^e_{\mathrm{source}}|$ represent the areas of the elements
containing the sink and the source, respectively.
First we approximated this problem using CGLS, GLS(Hdiv), MGLS and
HVM with a uniform mesh of $20\times20$  bilinear quadrilateral
elements.
The exact solution (\ref{eq:solfive}) and CGLS approximation for
potential are presented in Figures \ref{fig:fivesp-pote} and
\ref{fig:fivesp-cgls}, respectively. Similar results were obtained
with the other methods.
\begin{figure}[htb]
\centering
\subfloat[Exact.]{\includegraphics[angle=0,width=.49\textwidth,angle=0]{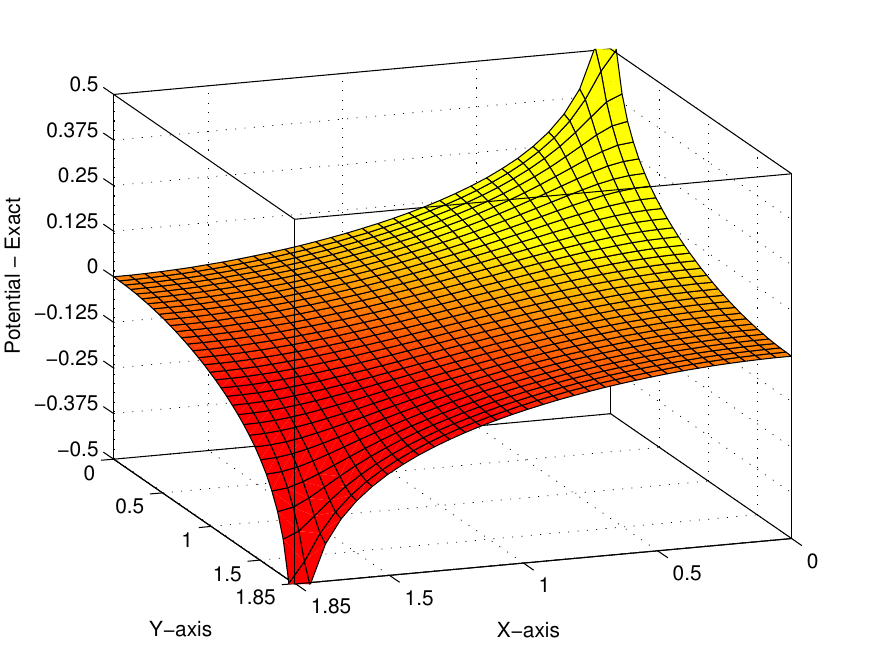}
\label{fig:fivesp-pote}}
%
%
\subfloat[CGLS.]{\includegraphics[angle=0,width=.49\textwidth,angle=0]{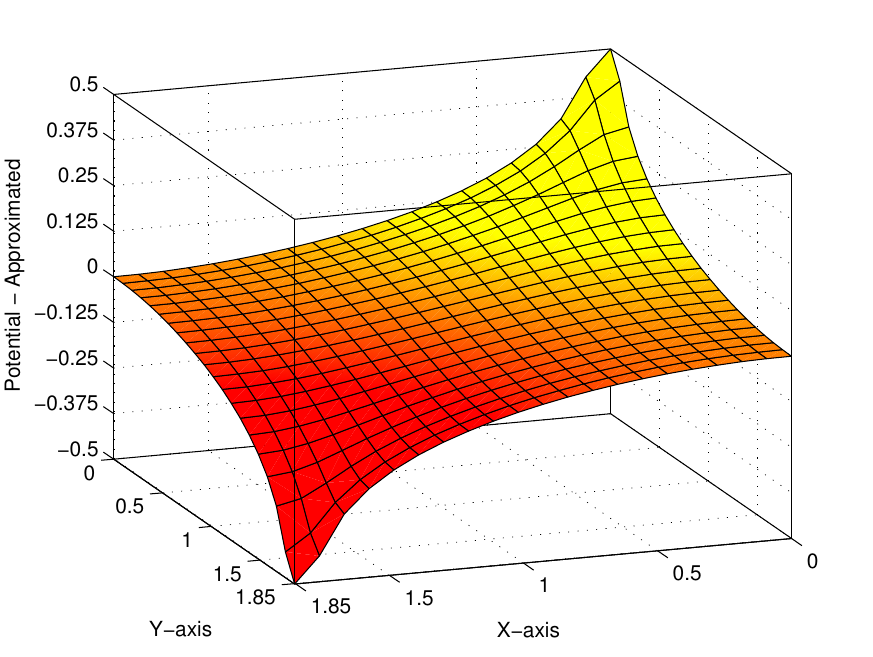}
    \label{fig:fivesp-cgls}}
\caption{Five-Spot Problem. Potential field: (a) Exact solution and (b) finite element approximation with
CLSG and a mesh of $20\times20$ bilinear elements.}
\label{fig:graph-fivep}
\end{figure}
In Figure \ref{fig:fivesp-vx-exata} we plot the exact solution for
the $u_x$ component of the velocity field. The approximations with
CGLS, GLS(Hdiv) and HVM for the $20\times20$ bilinear elements mesh
are shown in Figures \ref{fig:fivesp-vx-cgls},
\ref{fig:fivesp-vx-glshdiv} and \ref{fig:fivesp-vx-HVM},
respectively. We can see that they slightly differ in a region near
to the sources.
\begin{figure}[htb]
\centering
\subfloat[Exact.]{\includegraphics[angle=0,width=.49\textwidth,angle=0]{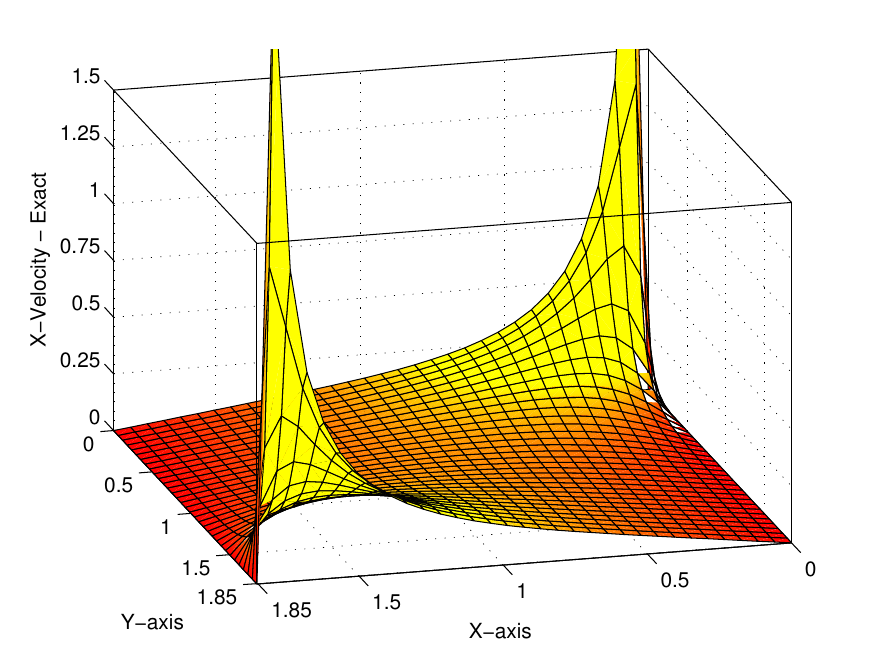}
\label{fig:fivesp-vx-exata}}
%
%
\subfloat[CGLS.]{\includegraphics[angle=0,width=.49\textwidth,angle=0]{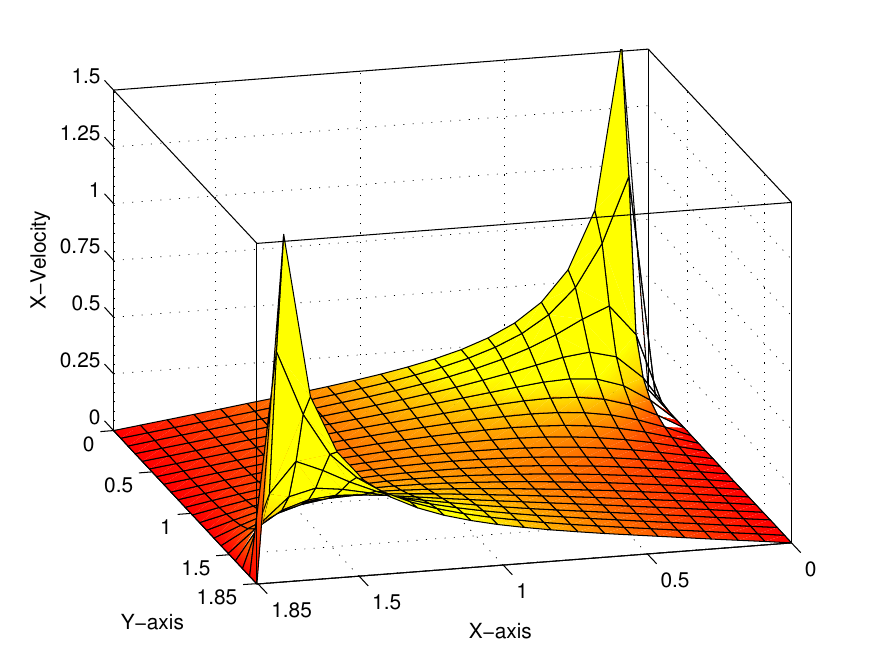}
\label{fig:fivesp-vx-cgls}}
\caption{$Five-Spot Problem. u_x$ Component of the velocity field: (a) Exact and (b)
CGLS approximation for a mesh of $20\times20$ bilinear elements.}
\label{fig:graph-fivev1}
\end{figure}
\begin{figure}[htb]
\centering
\subfloat[GLS(Hdiv).]{\includegraphics[angle=0,width=.49\textwidth,angle=0]{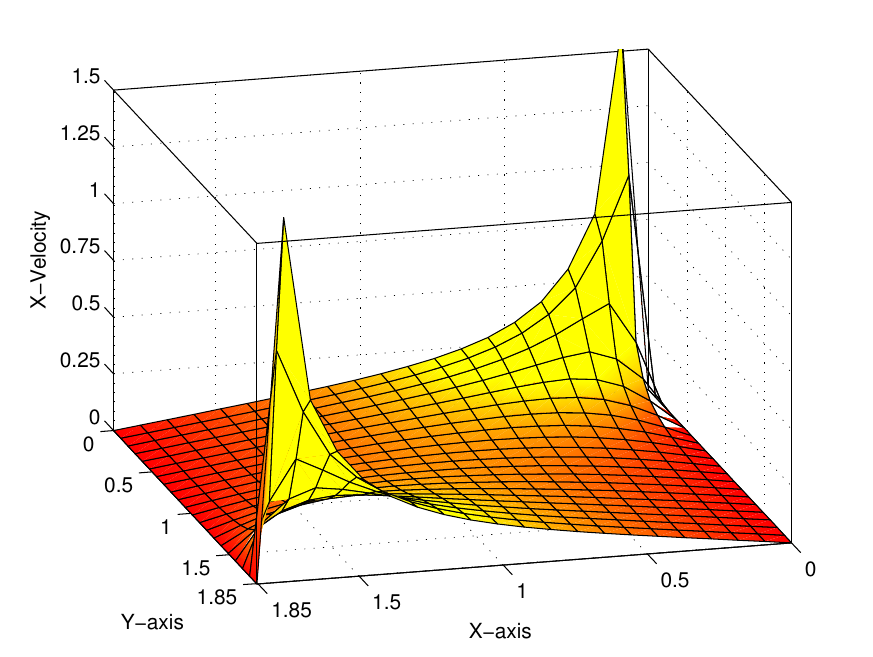}
\label{fig:fivesp-vx-glshdiv}}
%
%
\subfloat[HVM.]{\includegraphics[angle=0,width=.49\textwidth,angle=0]{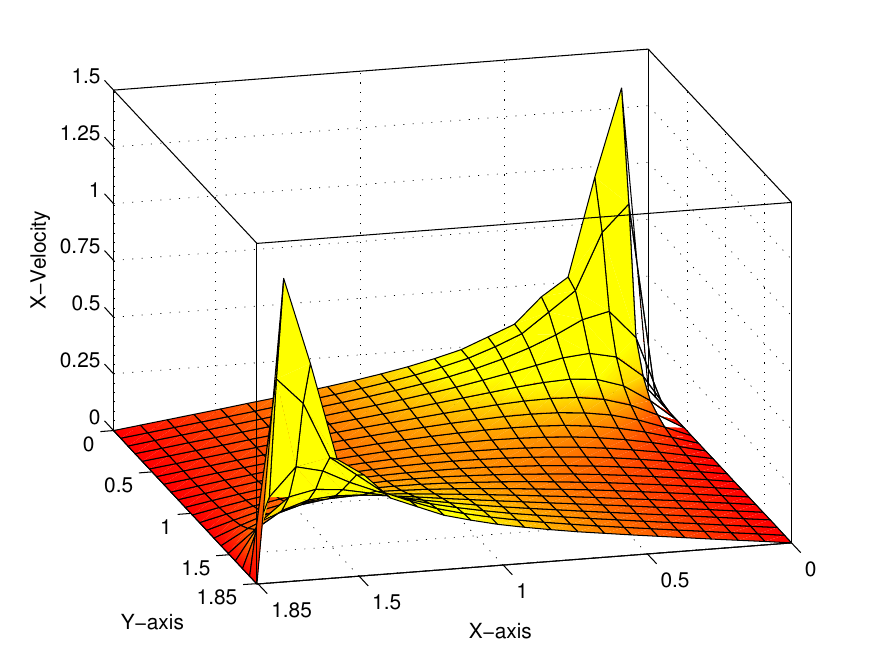}
\label{fig:fivesp-vx-HVM}}
\caption{Five-Spot Problem. $u_x$ Component of the velocity field: finite element approximation with (a) MGLS and (b) HVM for a mesh of $20\times20$ bilinear elements.}
\label{fig:graph-fivev2}
\end{figure}

We also perform a convergence study excluding the errors in two
small regions containing the singular loads consisting of a quarter
of a circle of radium $R=(0.1\sqrt{2})L$ centered at each point of
singularity. For the whole domain, not excluding the singularities
regions, we may expect that all methods achieve, at best, first
order rate of convergence for the pressure in $L^2(\Omega)$ norm,
and no convergence for the velocity approximation.
In this convergence study we adopt equal order interpolation for
velocity and pressure and uniform meshes with $16\times16$,
$32\times32$, $64\times64$  and $128\times128$ bilinear elements and
of $8\times8$, $16\times16$, $32\times32$ and $64\times64$
biquadratic elements.
The results for the $[L^2(\Omega)]^2$ norm of velocity using
Q$_1$Q$_1$ and Q$_2$Q$_2$ elements are shown in Figure
\ref{fig:ce-vel-five}. They show the robustness of the stabilized
methods preserving their accuracy away from small regions containing
the singularities.
Note the optimal convergence rates for velocity obtained by CGLS:
$O(h^{2.0})$ for bilinear and $O(h^{3.0})$ for biquadratic elements.
The results for potential, shown in Figure \ref{fig:ce-pot-five},
also confirm the predicted rates of convergence.
\begin{figure}[htb]
\centering
\subfloat[Q$_{1}$Q$_{1}$ Elements.]{\includegraphics[angle=0,width=.49\textwidth,angle=0]{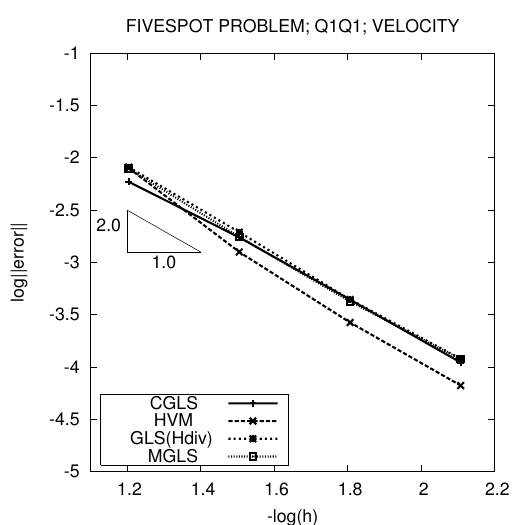}
\label{fig:ce-vel-five-q1}}
%
%
\subfloat[Q$_{2}$Q$_{2}$ Elements.]{\includegraphics[angle=0,width=.49\textwidth,angle=0]{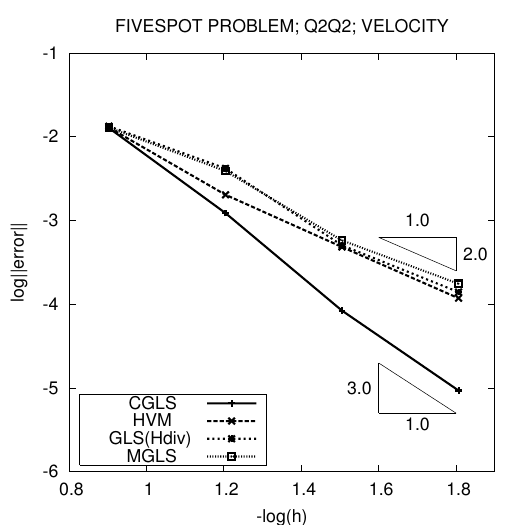}
\label{fig:ce-vel-five-q2}}
\caption{Five-Spot Problem. Convergence study for the regularized problem: Velocity.}
\label{fig:ce-vel-five}
\end{figure}
\begin{figure}[htb]
\centering
\subfloat[Q$_{1}$Q$_{1}$ Elements.]{\includegraphics[angle=0,width=.49\textwidth,angle=0]{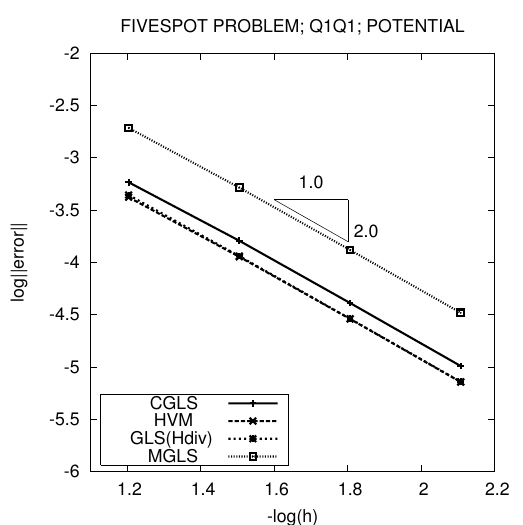}
\label{fig:ce-pot-five-q1}}
%
%
\subfloat[Q$_{2}$Q$_{2}$ Elements.]{\includegraphics[angle=0,width=.49\textwidth,angle=0]{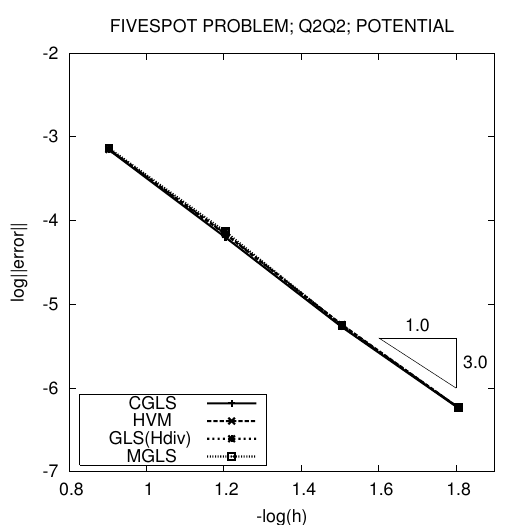}
\label{fig:ce-pot-five-q2}}
\caption{Five-Spot Problem. Convergence study for the regularized problem: Potential.}
\label{fig:ce-pot-five}
\end{figure}
%

%
%
\section{Concluding Remarks} \label{sec:conclu}
%
%

Unconditionally  stable mixed finite element methods for Darcy flow
were derived by combining least-squares residual forms of the the
governing equations with the classical mixed formulations as in the
mixed Petrov-Galerkin method or Galerkin Least Squares formulations.
Considering continuous Lagrangian interpolation of equal order for
velocity and potential fields, numerical analysis of the proposed
formulations lead the following conclusions:
\begin{itemize}

 \item Subtracting to the Primal mixed formulation only the least
squares residual of Darcy's law, a finite element method stable in
$\left[L^2(\Omega)\right]^n \times H^1(\Omega)$ is derived  with
optimal rates of convergency for potential in $H^1(\Omega)$ and
quasi optimal rates for velocity in $\left[L^2(\Omega)\right]^n$
norms. This method is shown to be equivalent to the HVM method,
presented in \cite{MASUD2002}, and also to a stabilized method
obtained by combining two classical Galerkin formulations: Dual
mixed in velocity and potential and the Primal formulation, in
pressure only.

\item
Combining least square residuals of both conservation of mass and
Darcy's law with the Dual mixed formulation, a finite element method
stable in ${H(\div,\Omega)} \\\times H^1(\Omega)$ is obtained with optimal rates
of convergence for velocity in ${H(\div,\Omega)}$ and potential in $H^1(\Omega)$
norms.

\item Including a least square residual of the curl of Darcy's
law  an $\left[H^1(\Omega)\right]^n\times H^1(\Omega)$ stable method
is derived, with optimal rates of convergency for both velocity and
potential in the norm of the product space. Under regularity
assumptions, an optimal estimate is proved for this method in the
$\left[L^2(\Omega)\right]^2\times L^2(\Omega)$ norm.

\end{itemize}

Other possible interpolations for the velocity in the GLS(Hdiv) and
HVMS methods and a Stokes compatible formulation for Darcy flow are
also commented.
The  numerical experiments  performed, illustrate the flexibility
and robustness of the proposed finite element formulations and
confirm the predicted rates of convergence.

\section*{Acknowledgements}

The authors thank to  Funda\c c\~ao Carlos Chagas Filho de Amparo \`a
Pesquisa do Estado do Rio de Janeiro (FAPERJ) and tothe National Council for Scientific and Technological Development (CNPq)
for the sponsoring.

\bibliographystyle{elsart-num}
\bibliography{tese}

\begin{thebibliography}{10}
\expandafter\ifx\csname url\endcsname\relax
  \def\url#1{\texttt{#1}}\fi
\expandafter\ifx\csname urlprefix\endcsname\relax\def\urlprefix{URL }\fi

\bibitem{RAVIART77}
P.~A. Raviart, J.~M. Thomas, A mixed finite element method for second order
  elliptic problems, in: Math. Aspects of the F.E.M., No. 606 in Lectute Notes
  in Mathematics, Springer-Verlag, 1977, pp. 292--315.

\bibitem{FORTIN91}
F.~Brezzi, M.~Fortin, Mixed and Hybrid Finite Element Methods, Vol.~15 of
  Springer Series in Computational Mathematics, Springer-Verlag, New York,
  1991.

\bibitem{LOULA95}
A.~F.~D. Loula, F.~A. Rochinha, M.~A. Murad, Higher-order gradient
  post-processings for second-order elliptic problems, Computer Methods in
  Applied Mechanics and Engineering 128 (1995) 361--381.

\bibitem{CORREA2007}
M.~R. Correa, A.~F.~D. Loula, Stabilized velocity post-processings for {D}arcy
  flow in heterogenous porous media, Communications in Numerical Methods in
  Engineering 23 (2007) 461--489.

\bibitem{CORDES92}
C.~Cordes, W.~Kinzelbach, Continuous groundwater velocity fields and path lines
  in linear, bilinear and trilinear finite elements, Water Resources Research
  28~(11) (1992) 2903--2911.

\bibitem{DURLOFSKY94}
L.~J. Durlofsky, Accuracy of mixed and control volume finite element
  approximations to {D}arcy velocity and related quantities, Water Resources
  Research 30~(4) (1994) 965--973.

\bibitem{BREZZI74}
F.~Brezzi, On the existence, uniqueness and approximation of saddle point
  problems arising from lagrange multipliers, RAIRO Analyse
  num\'erique/Numerical Analysis 8(R-2) (1974) 129--151.

\bibitem{LOULA87A}
A.~F.~D. Loula, T.~J.~R. Hughes, L.~P. Franca, Mixed {P}etrov-{G}alerkin
  methods for the {T}imoshenko beam problem, Computer Methods in Applied
  Mechanics and Engineering 63 (1987) 133--154.

\bibitem{LOULA88}
A.~F.~D. Loula, E.~M. Toledo, Dual and Primal Mixed Petrov-Galerkin Finite
  Element Methods in Heat Transfer Problems, LNCC - Technical Report 048/88,
  1988.

\bibitem{FRANCA88}
L.~P. Franca, T.~J.~R. Hughes, Two classes of mixed finite element methods,
  Computer Methods in Applied Mechanics and Engineering 69 (1988) 89--129.

\bibitem{MASUD2002}
A.~Masud, T.~J.~R. Hughes, A stabilized finite element method for {D}arcy flow,
  Computer Methods in Applied Mechanics and Engineering 191 (2002) 4341--4370.

\bibitem{BREZZI2005}
F.~Brezzi, T.~J.~R. Hughes, L.~D. Marini, A.~Masud, A mixed discontinuous
  {G}alerkin method for {D}arcy flow, SIAM J. Scientific Comput. 22-23 (2005)
  119--145.

\bibitem{HUGHES2006A}
T.~J.~R. Hughes, A.~Masud, J.~Wan, A stabilized mixed discontinuous {G}alerkin
  method for {D}arcy flow, Computer Methods in Applied Mechanics and
  Engineering 195 (2006) 3347--3381.

\bibitem{LOULA2006A}
A.~F.~D. Loula, M.~R. Correa, Numerical analysis of stabilized mixed finite
  element methods for {D}arcy flow, in: III European Conference on
  Computational Mechanics -- ECCM 2006, Lisbon, Portugal, 2006.

\bibitem{CORREA2006}
M.~R. Correa, Stabilized Finite Element Methods for {D}arcy and Coulped
  {S}tokes-{D}arcy Flows, D.Sc. Thesis, LNCC, Petr\'opolis, RJ, Brazil (in
  Portuguese), 2006.

\bibitem{BARRENECHEA2007}
G.~Barrenechea, L.~P. Franca, F.~Valentin, A {P}etrov-{G}alerkin enriched
  method: A mass conservative finite element method for the {D}arcy equation,
  Computer Methods in Applied Mechanics and Engineering 196 (2007) 2449--2464.

\bibitem{CIARLET78}
P.~G. Ciarlet, The Finite Element Method for Elliptic Problems, Studies in
  Mathematics and its Applications, North-Holland Publishing Company, 1978.

\bibitem{ZIENK92}
O.~C. Zienkiewicz, J.~Z. Zhu, The superconvergent patch recovery and a
  posteriori error estimates. part 1: The recovery technique., International
  Journal for Numerical Methods in Engineering 33 (1992) 1331--1364.

\bibitem{ODEN4}
G.~F. Carey, J.~T. Oden, Finite Elements: Mathematical Aspects, Volume 4,
  Prentice Hall, Inc., Eaglewood Cliffs, New Jersey, 1983.

\bibitem{GIRAULT86}
V.~Girault, P.~Raviart, Finite Element Methods for Navier-Stokes Equations:
  Theory and Algorithms., Springer Series in Computational Mathematics,
  Springer-Verlag, 1986.

\bibitem{BREZZI85}
F.~Brezzi, J.~{Douglas Jr.}, L.~D. Marini, Two families of mixed finite
  elements for second order elliptic problems, Numerische Mathematik 47 (1985)
  217--235.

\bibitem{CAI97}
Z.~Cai, T.~A. Manteuffel, S.~F. Mccormick, First-order system least squares for
  second-order partial differential equations: {P}art ii, SIAM Journal on
  Numerical Analysis 34~(2) (1997) 425--454.

\bibitem{BABUSKA71A}
I.~Bab{\v u}ska, Error bounds for finite element method, Numerische Mathematik
  16 (1971) 322--133.

\bibitem{GELFAND63}
I.~M. Gelfand, S.~V. Fomin, Calculus of Variations, Prentice-Hall, 1963.

\bibitem{NAKSHATRALA2006}
K.~B. Nakshatrala, D.~Z. Turner, K.~D. Hjelmstad, A.~Masud, A stabilized finite
  element method for {D}arcy flow based on a multiescale decomposition of the
  solution, Computer Methods in Applied Mechanics and Engineering 195 (2006)
  4036--4049.

\bibitem{VERFURTH84}
R.~Verf\"urth, Error estimates for a mixed finite element approximation of the
  {S}tokes equations, RAIRO Analyse num\'erique/Numerical Analysis 18~(2)
  (1984) 175--182.

\bibitem{DISCACCIATI2002}
M.~Discacciati, E.~Miglio, A.~Quarteroni, Mathematical and numerical models for
  coupling surface and growndwater flows, Applied Numerical Mathematics 43
  (2002) 57--74.

\bibitem{LAYTON2003}
W.~J. Layton, F.~Schieweck, I.~Yotov, Coupling fluid flow with porous media
  flow, SIAM Journal on Numerical Analysis 40~(6) (2003) 2195--2218.

\bibitem{RIVIERE2005B}
B.~Rivi\`ere, Analysis of a discontinuous finite element method for the coupled
  {S}tokes and {D}arcy problems, Journal of Scientific Computing 22-23 (2005)
  479--500.

\bibitem{ARBOGAST2007}
T.~Arbogast, D.~S. Brunson, A computational method for approxima\-ting a
  {D}arcy-{S}tokes system governing a vuggy porous medium, Computational
  Geosciences 11~(3) (2007) 207--218.

\bibitem{CORREA2007A}
M.~R. Correa, A.~Loula, An adjoint stabilized mixed finite element method for
  porous media flow (in {P}ortuguese), in: XXVIII CILAMCE -- Iberian
  Latin-American Congress on Computational Methods in Engineering, Porto,
  Portugal, 2007.

\bibitem{DOUGLAS89}
J.~{Douglas Jr}., J.~Wang, An absolutely stabilized finite element method for
  the {S}tokes problem, Mathematics of Computation 52~(186) (1989) 495--508.

\bibitem{MOREL65}
H.~J. Morel-Seytoux, Analitical numerical method in waterflooding predictions,
  SPE J. 5 (1965) 247--285.

\end{thebibliography}







\end{document}